\documentclass[12pt]{article}
\usepackage{graphicx}
\usepackage{amsmath,amsthm,amssymb,amsfonts,euscript,enumerate}
\usepackage{amsthm}
\usepackage{color,wrapfig}
\usepackage{pxfonts}
\usepackage{verbatim}
\newtheorem{thm}{Theorem}[section]

\newtheorem{lem}[thm]{Lemma}
\newtheorem{prop}[thm]{Proposition}

\begin{document}
\title{Super fast vanishing solutions of the fast diffusion equation}
\author{Shu-Yu Hsu\\
Department of Mathematics\\
National Chung Cheng University\\
168 University Road, Min-Hsiung\\
Chia-Yi 621, Taiwan, R.O.C.\\
e-mail: shuyu.sy@gmail.com}
\date{April 7, 2020}
\smallbreak \maketitle
\begin{abstract}
We will extend a recent result of B.~Choi, P.~Daskalopoulos and J.~King \cite{CDK}. For any $n\ge 3$, $0<m<\frac{n-2}{n+2}$ and $\gamma>0$, we will construct subsolutions and supersolutions of the fast diffusion equation $u_t=\frac{n-1}{m}\Delta u^m$ in $\mathbb{R}^n\times (t_0,T)$, $t_0<T$, which decay at the rate $(T-t)^{\frac{1+\gamma}{1-m}}$ as $t\nearrow T$. As a consequence we obtain the existence of unique solution of the Cauchy problem $u_t=\frac{n-1}{m}\Delta u^m$ in $\mathbb{R}^n\times (t_0,T)$, $u(x,t_0)=u_0(x)$ in $\mathbb{R}^n$, which decay at the rate $(T-t)^{\frac{1+\gamma}{1-m}}$ as $t\nearrow T$ when $u_0$ satisfies appropriate decay condition.
\end{abstract}

\vskip 0.2truein



\vskip 0.2truein
\setcounter{section}{0}

\section{Introduction}
\setcounter{equation}{0}
\setcounter{thm}{0}

Recently there is a lot of interest in the following singular diffusion
equation \cite{A}, \cite{DK}, \cite{P}, \cite{V2},
\begin{equation}\label{fde-eqn}
u_t=\frac{n-1}{m}\Delta u^m\quad\mbox{ in }\mathbb{R}^n\times (t_0,T)
\end{equation}
which arises in the study of many physical models and geometric flows. When $0<m<1$,
\eqref{fde-eqn} is called the fast diffusion equation.  As observed by S.~Brendle, P.~Daskalopoulos, M.~del Pino, J.~King, M.~S\'aez, N.~Sesum, and others \cite{B1}, \cite{B2}, \cite{DPKS1}, \cite{DPKS2}, \cite{PS}, the metric $g=u^{\frac{4}{n-2}}dy^2$ satisfies the Yamabe flow
\begin{equation}\label{log1-eqn}
\frac{\partial g}{\partial t}=-Rg
\end{equation}
on $\mathbb{R}^n$, $n\ge 3$, for $0<t<T$, where $R$ is the scalar curvature of the metric $g$, if and only if $u$ satisfies \eqref{fde-eqn} with
\begin{equation*}
m=\frac{n-2}{n+2}.
\end{equation*}

As observed by L.~Peletier \cite{P} and J.L~Vazquez \cite{V1}  the behaviour of the solutions of \eqref{fde-eqn} for the cases $m>1$, $\frac{(n-2)_+}{n}<m<1$ and $0<m<\frac{(n-2)_+}{n}$ varies a lot. When $m>1$, any solution of  
\begin{equation}\label{Cauchy-problem}
\left\{\begin{aligned}
u_t=&\frac{n-2}{m}\Delta u^m, u\ge 0,\quad\mbox{ in }\mathbb{R}^n\times (t_0,T)\\
u(x,t_0)=&u_0\qquad\qquad\qquad\quad\,\mbox{ in }\mathbb{R}^n
\end{aligned}\right.
\end{equation}
will have compact support for any time $t_0<t<T$ provided $0\le u_0\in L^{\infty}(\mathbb{R}^n)$ has compact support (\cite{A}). On the other hand  when $\frac{(n-2)_+}{n}<m<1$, M.A.~Herrero and M.~Pierre \cite{HP} proved the global existence and uniqueness of positive solution of \eqref{Cauchy-problem}  for any $0\lvertneqq u_0\in L_{loc}^1(\mathbb{R}^n)$.
For
\begin{equation}\label{m-range}
n\ge 3\quad\mbox{ and }\quad 0<m<\frac{n-2}{n}, 
\end{equation}
it was observed by P.~Daskalopoulos, Galaktionov, L.A.~Peletier, M.~del Pino and N.~Sesum etc. (\cite{DS}, \cite{DKS}, \cite{GP}, \cite{PS}) that \eqref{Cauchy-problem} has positive solutions which vanish in a finite time $T$ when 
$0\le u_0\in L^{\infty}(\mathbb{R}^n)$ satisfies 
\begin{equation*}
u_0(x)\le C|x|^{-\frac{2}{1-m}}\quad\mbox{ as }|x|\to\infty
\end{equation*}
for some constant $C>0$. Moreover the finite time extinction solutions of \eqref{Cauchy-problem} considered in these papers all decay at the rate $(T-t)^{\frac{1}{1-m}}$ near the extinction time $T>0$. 
 
On the other hand it was proved by S.Y.~Hsu in \cite{Hs2} that when $n\ge 3$, $0<m\le (n-2)/n$, and
$0\le u_0\in L_{loc}^p(\mathbb{R}^n)$ for some constant $p$ satisfying 
$p>(1-m)n/2$ and 
\begin{equation*}
\liminf_{R\to\infty}\frac{1}{R^{n-\frac{2}{1-m}}}\int_{|x|\le R}u_0\,dx
=\infty,
\end{equation*}
\eqref{Cauchy-problem} has a unique global positive solution. Asymptotic large time behaviour of global solution of \eqref{Cauchy-problem} when \eqref{m-range} holds and $u_0$ also satisfies $u_0(x)\approx A|x|^{-q}$ as $|x|\to\infty$ for some 
constants $A>0$, $q<n/p$, was also proved by S.Y.~Hsu in \cite{Hs2}. Asymptotic large time behaviour of global solution of \eqref{Cauchy-problem} when $n\ge 3$, $m=\frac{n-2}{n+2}$, and 
\begin{equation*}
u_0(x)\approx\left(\frac{(n-1)(n-2)}{\beta|x|^2}\log |x|\right)^{\frac{1}{1-m}}\quad\mbox{ as }|x|\to\infty
\end{equation*}
for some constant $\beta>0$ was also proved by B.~Choi and P.~Daskalopoulos in \cite{CD}. When $n\ge 3$, $0<m<\frac{n-2}{n}$, $m\ne \frac{n-2}{n+2}$ and 
\begin{equation*}
u_0(x)\approx\left(\frac{2(n-1)(n-2-nm)}{\beta (1-m)|x|^2}\log |x|\right)^{\frac{1}{1-m}}\quad\mbox{ as }|x|\to\infty
\end{equation*}
for some constant $\beta>0$, asymptotic large time behaviour of global solution of \eqref{Cauchy-problem} was proved by S.Y.~Hsu in \cite{Hs5}. First order symptotic behaviour of the  self-similar solutions of \eqref{fde-eqn} when $n\ge 3$, $0<m<\frac{n-2}{n}$, was proved by S.Y.~Hsu in \cite{Hs1}, \cite{Hs3}, \cite{Hs4}, using integral equation technique. Second order asymptotic behaviour of the self-similar solutions of \eqref{fde-eqn} when $n\ge 3$, $m=\frac{n-2}{n+2}$, was proved by P.~Daskalopoulos, J.~King and N.~Sesum \cite{DKS}. Second order asymptotic behaviour of the self-similar solutions of \eqref{fde-eqn} when $n\ge 3$, $0<m<\frac{n-2}{n}$, was proved by  B.~Choi, P.~Daskalopoulos, S.Y.~Hsu, K.M.~Hui and Soojung Kim \cite{CD}, \cite{Hs5}, \cite{HK}.

In the recent paper \cite{CDK} of B.~Choi, P.~Daskalopoulos and J.~King they proved that for any $n\ge 3$, $m=\frac{n-2}{n+2}$ and $\gamma>0$,  there  exist finite time extinction solution of \eqref{Cauchy-problem} which decay at the rate $(T-t)^{\frac{1+\gamma}{1-m}}$ near the extinction time $T>0$ when $u_0$ satisfies appropriate decay condition. They also proved the behaviour of such solutions near the extinction time and showed that such solutions have type II singularities near the extinction time. 

In this paper we will extend their results. 
For any $n\ge 3$, $0<m<\frac{n-2}{n+2}$ and $\gamma>0$, we will construct subsolutions and supersolutions of \eqref{fde-eqn} which decay at the rate $(T-t)^{\frac{1+\gamma}{1-m}}$ as $t\nearrow T$. As a consequence we obtain the existence of unique solution of the Cauchy problem \eqref{Cauchy-problem} which decay at the rate $(T-t)^{\frac{1+\gamma}{1-m}}$ as $t\nearrow T$ when $u_0$ satisfies appropriate decay condition. 

We will use a modification of the technique of \cite{CDK} to construct subsolutions and supersolutions of \eqref{fde-eqn} using match asymptotic technique gluing some particular inner subsolutions (supersolutions respectively) and outer subsolutions (supersolutions, respectively) of \eqref{fde-eqn}. These subsolutions and supersolutions  of \eqref{fde-eqn} will then be used as barriers for constructing the unique solution of \eqref{Cauchy-problem}  when $u_0$ decays at the rate $(T-t)^{\frac{1+\gamma}{1-m}}$ near the extinction time $T>0$.  

Unless stated otherwise we will let $n$ and $m$ satisfy \eqref{m-range}
and 
\begin{equation*}
m\ne\frac{n-2}{n+2}
\end{equation*}
for the rest of the paper. Suppose $u$ is a radially symmetric solution of \eqref{fde-eqn} in $\mathbb{R}^n\times (0,T)$. Let
\begin{equation}\label{w-defn}
w(s,t)=r^2u(r,t)^{1-m},\quad s=\log r,\quad r=|x|, \quad x\in\mathbb{R}^n.
\end{equation}
Then $w$ satisfies
\begin{equation*}
(w^{\frac{1}{1-m}})_t=\frac{n-1}{m}\left\{(w^{\frac{m}{1-m}})_{ss}+\left(\frac{n-2-m(n+2)}{1-m}\right)(w^{\frac{m}{1-m}})_s-\frac{2m(n-2-nm)}{(1-m)^2}w^{\frac{m}{1-m}}\right\}\,\,\mbox{in }\mathbb{R}\times (0,T)
\end{equation*}
or equivalently
\begin{equation}\label{w-eqn2}
w_t=(n-1)\left\{\frac{w_{ss}}{w}+\left(\frac{2m-1}{1-m}\right)\frac{w_s^2}{w^2}+\left(\frac{n-2-m(n+2)}{1-m}\right)\frac{w_s}{w}-\frac{2(n-2-nm)}{(1-m)}\right\}\quad\mbox{in }\mathbb{R}\times (0,T).
\end{equation}
Let $\gamma>0$ and 
\begin{equation}\label{hat-w-defn}
\hat{w}(\eta,\tau)=(T-t)^{-1}w(s,t),\quad\eta=(T-t)^{\gamma}s,\quad\tau=-\log\, (T-t).
\end{equation}
By \eqref{w-eqn2} and a direct computation $\hat{w}$ satisfies
\begin{equation}\label{l0-eqn0}
L_0(\hat{w})=0 
\end{equation}
in $\mathbb{R}\times (-\log T,\infty)$ where
\begin{align}\label{l0-defn}
L_0(\hat{w}):=&\hat{w}_{\tau}-(n-1)\left\{e^{-2\gamma\tau}\left(\frac{\hat{w}_{\eta\eta}}{\hat{w}}+\left(\frac{2m-1}{1-m}\right)\frac{\hat{w}_{\eta}^2}{\hat{w}^2}\right)
+\left(\frac{n-2-m(n+2)}{1-m}\right)e^{-\gamma\tau}\frac{\hat{w}_{\eta}}{\hat{w}}\right\}\notag\\
&\qquad -\left(\gamma\eta\hat{w}_{\eta}+\hat{w}-\frac{2(n-1)(n-2-nm)}{1-m}\right)\qquad\quad\mbox{ in }\mathbb{R}\times (-\log\, T,\infty).
\end{align}
Let $A>0$ and 
\begin{equation}\label{phi0-defn}
\phi_0(\eta)=a_0\left(1-A^{\frac{1}{\gamma}}\eta^{-\frac{1}{\gamma}}\right)\quad\forall\eta>A.
\end{equation}
where
\begin{equation}\label{a0-defn}
a_0=\frac{2(n-1)(n-2-nm)}{(1-m)}.
\end{equation} 
Then $\phi_0$ is positive in $(A,\infty)$ and satisfies 
\begin{equation*}
\gamma\eta\phi_{0,\eta}+\phi_0-\frac{2(n-1)(n-2-nm)}{1-m}=0\quad\mbox{ in }(A,\infty).
\end{equation*}
Hence $\phi_0$ can be regarded as a limiting first order approximate solution of \eqref{l0-eqn0} as $\tau\to\infty$. Let
\begin{equation*}
E_1:=e^{-2\gamma\tau}\left(\frac{\hat{w}_{\eta\eta}}{\hat{w}}+\left(\frac{2m-1}{1-m}\right)\frac{\hat{w}_{\eta}^2}{\hat{w}^2}\right)
+e^{-\gamma\tau}\left(\frac{n-2-m(n+2)}{1-m}\right)\frac{\hat{w}_{\eta}}{\hat{w}}.
\end{equation*}
Then
\begin{equation*}
E_1=(e^{\gamma\tau}\hat{w})^{-2}\left(\hat{w}\hat{w}_{\eta\eta}+\left(\frac{2m-1}{1-m}\right)\hat{w}_{\eta}^2\right)+(e^{\gamma\tau}\hat{w})^{-1}\left(\frac{n-2-m(n+2)}{1-m}\right)\hat{w}_{\eta}.
\end{equation*}
Hence by assuming the boundedness of $\hat{w}$, $\hat{w}_{\eta}$ and  $\hat{w}_{\eta\eta}$, the  term $E_1$ in \eqref{l0-defn} is negligible in the space-time region 
\begin{equation}\label{outer-region}
(e^{\gamma\tau}\hat{w}(\eta,\tau))^{-1}=o(1)\quad\mbox{ as }\tau\to\infty.
\end{equation}
This suggest that the domain is divided into the outer region given by \eqref{outer-region} in which the diffusion and advection terms of the equation \eqref{l0-defn} are negligible and inner region given by 
\begin{equation}\label{inner-region}
e^{\gamma\tau}\hat{w}(\eta,\tau)=O(1)\quad\mbox{ as }\tau\to\infty
\end{equation}
in which the diffusion and advection terms of the equation \eqref{l0-defn} are not negligible. This suggests the transformation
\begin{equation}\label{w-overline-defn}
\overline{w}(\xi,\tau)=e^{\gamma\tau}\hat{w}(\eta,\tau),\quad\eta=A+e^{-\gamma\tau}\xi.
\end{equation} 
Then by \eqref{hat-w-defn} and \eqref{w-overline-defn},
\begin{align}
&\overline{w}(\xi,\tau)=e^{(1+\gamma)\tau}w(\xi+Ae^{\gamma \tau},t)\label{w-overline1}\\
\Rightarrow\quad&w(s,t)=e^{-(1+\gamma)\tau}\overline{w}(s-Ae^{\gamma\tau},\tau)\label{w-w2-relation}\\
\Rightarrow\quad&w_t=e^{-\gamma\tau}(\overline{w}_{\tau}-(1+\gamma)\overline{w})-A\gamma\overline{w}_{\xi}.\label{wt-w2-relation}
\end{align}
By \eqref{w-eqn2}, \eqref{w-w2-relation} and  \eqref{wt-w2-relation},
\begin{equation}\label{L1-eqn}
L_1(\overline{w})=0 
\end{equation}
in $\mathbb{R}\times (-\log T,\infty)$  where
\begin{align}\label{l1-defn}
L_1(\overline{w}):=&e^{-\gamma\tau}(\overline{w}_{\tau}-(1+\gamma)\overline{w})-(n-1)\left\{\frac{\overline{w}_{\xi\xi}}{\overline{w}}+\left(\frac{2m-1}{1-m}\right)\frac{\overline{w}_{\xi}^2}{\overline{w}^2}+\left(\frac{n-2-m(n+2)}{1-m}\right)\frac{\overline{w}_{\xi}}{\overline{w}}\right\}\notag\\
&\qquad +\frac{2(n-1)(n-2-nm)}{1-m}-\gamma A\overline{w}_{\xi}.
\end{align}
Note that by \eqref{w-defn} and \eqref{w-overline1},
\begin{equation*}
\overline{w}(\xi,\tau)=(T-t)^{-\gamma-1}r^2u(r,t)^{1-m}
\end{equation*}
with
\begin{equation*}
\quad\xi=\log r-A(T-t)^{-\gamma}, \tau=-\log\, (T-t)
\end{equation*}
or equivalently
\begin{equation*}
u(x,t)=\left(\frac{(T-t)^{1+\gamma}}{|x|^2}\overline{w}(\xi,\tau)\right)^{\frac{1}{1-m}}.
\end{equation*}
Let $\lambda>0$ and $v_0$ be the unique radially symmetric solution of 
\begin{equation}\label{elliptic-eqn}
\left\{\begin{aligned}
&\frac{n-1}{m}\Delta v^m+\frac{2\gamma A}{1-m} v+\gamma A x\cdot\nabla v=0,\quad v>0,\quad\mbox{ in }\mathbb{R}^n\\
&v(0)=\lambda\end{aligned}\right.
\end{equation}
given by Theorem 1.1 of \cite{Hs1} and 
\begin{equation}\label{phi0-bar-defn}
\bar{\phi}_0(s)=e^{2s}v_0(e^s)^{1-m}.
\end{equation}
Then by (3.4) of \cite{Hs1} $\bar{\phi}_0$ satisfies 
\begin{equation}\label{phi2-0-eqn}
(n-1)\left\{\frac{\bar{\phi}_{0,ss}}{\bar{\phi}_0}+\left(\frac{2m-1}{1-m}\right)\frac{\bar{\phi}_{0,s}^2}{\bar{\phi}_0^2}+\left(\frac{n-2-m(n+2)}{1-m}\right)\frac{\bar{\phi}_{0,s}}{\bar{\phi}_0}\right\}-\frac{2(n-1)(n-2-nm)}{1-m}+\gamma A\bar{\phi}_{0,s}=0
\end{equation}
in $\mathbb{R}$. Hence $\bar{\phi}_0$ may be considered as a limiting first order approximate stationary solution of \eqref{L1-eqn} as $\tau\to\infty$. By adding some correction terms to the functions $\phi_0$ and $\bar{\phi}_0$ we will construct subsolutions and supersolutions of \eqref{l0-eqn0} and \eqref{L1-eqn}
respectively in the outer region \eqref{outer-region} and in the inner region \eqref{inner-region} respectively.

The plan of the paper is as follows. In section two we will construct subsolutions and supersolutions of \eqref{l0-eqn0} in the outer region. In section three we will construct subsolutions and supersolutions of \eqref{L1-eqn} in the inner region using match asymptotic method.
In section four we will construct distributional subsolutions and supersolutions of \eqref{fde-eqn} and we will use these as barriers to construct the unique solution of \eqref{Cauchy-problem}.

We start with some definitions. For any open set $O\in\mathbb{R}^n\times (0,T)$ we say that a positive function $u$ on $O$ is a solution (subsolution, supersolution, respectively) of \eqref{fde-eqn} in $O$ if $u\in C^{2,1}(O)$ satisfies
\begin{equation*}
\Delta u^m= u_t\quad\mbox{ in }O
\end{equation*}
($\ge$, $\le$, respectively) in the classical sense. For any $0\le u_0\in L_{loc}^1(\mathbb{R}^n)$ we say that  $u$ is a solution (subsolution, supersolution, respectively) of \eqref{Cauchy-problem} if $u$ is a solution (subsolution, supersolution, respectively) of \eqref{fde-eqn} in $\mathbb{R}^n\times (t_0,T)$ and satisfies
\begin{equation}
\|u(\cdot,t)-u_0\|_{L^1(K)}\to 0\quad\mbox{ as }t\to t_0^+
\end{equation}
for any compact subset $K$ of $\mathbb{R}^n$. We say that a function $u$ on $O$ is a weak solution (subsolution, supersolution, respectively) of \eqref{fde-eqn} in $O$ if $0\le u\in C(O)$ satisfies
\begin{equation*}
\iint_O\left(uf_t+\frac{n-1}{m}u^m\Delta f\right)\,dx\,dt
= 0
\end{equation*} 
($\ge$, $\le$, respectively) for any $f\in C_0^{\infty}(O)$. We say that a function $\hat{w}$ is a solution (subsolution, supersolution, respectively) of \eqref{l0-eqn0} in $O$ if $\hat{w}\in C^{2,1}(O)$ satisfies
\begin{equation*}
L_0(\hat{w})= 0 \quad\mbox{ in }O
\end{equation*}
($\le$, $\ge$, respectively) in the classical sense. Similarly we say that a function $\overline{w}$ is a solution (subsolution, supersolution, respectively) of \eqref{L1-eqn} in $O$ if $\overline{w}\in C^{2,1}(O)$ satisfies
\begin{equation*}
L_1(\overline{w})= 0 \quad\mbox{ in }O
\end{equation*}
($\le$, $\ge$, respectively) in the classical sense.

Let $t_2>t_1$, $R>0$, $B_R=\{x\in\mathbb{R}^n:|x|<R\}$ and $Q_R=B_R\times (t_1,t_2)$. Let $0\le g\in C((\partial B_R\times [t_1,t_2))\cup (\overline{B_R}\times \{t_1\}))$. We say that a function $\zeta$ on $\overline{Q_R}$ is a weak solution (subsolution, supersolution, respectively) of 
\begin{equation}\label{bd-domain-problem}
\left\{\begin{aligned}
&\frac{\partial\zeta}{\partial t}=\frac{n-1}{m}\Delta \zeta^m\quad\mbox{ in }Q_R\\
&\zeta(x,t)=g(x,t)\quad\mbox{ on }\partial B_R\times (t_2,t_1)\\
&\zeta(x,t_1)=g(x,t_1)\quad\mbox{ on }B_R
\end{aligned}\right.
\end{equation} 
if $0\le\zeta\in C([t_1,t_2];L^1(B_R))\cap L^{\infty}(Q_R)$ satisfies
\begin{align*}
\iint_{Q_R}\left(\zeta f_t+\frac{n-1}{m}\zeta^m\Delta f\right)\,dx\,dt
=&\frac{n-1}{m}\int_{t_1}^{t_2}\int_{\partial B_R}g^m\frac{\partial f}{\partial n}\,d\sigma\,dt+\int_{B_R}\zeta (x,t_2)f(x,t_2)\,dx\\
&\qquad -\int_{B_R}g(x,t_1)f(x,t_1)\,dx
\end{align*}
($\ge$, $\le$, respectively) for any $f\in C^{\infty}(\overline{B_R(0)}\times [t_1,t_2])$ which vanishes on $\partial B_R\times [t_1,t_2]$ where $\partial/\partial n$ is the dervative with respect to the unit outward normal $n$ on $\partial B_R$.

\section{Subsolutions and supersolutions in the outer region}
\setcounter{equation}{0}
\setcounter{thm}{0}

In this section we will construct subsolutions and supersolutions of \eqref{l0-eqn0} in the outer region.
Note that
\begin{equation}\label{L0-phi0-eqn}
L_0(\phi_0)=-(n-1)\left\{e^{-2\gamma\tau}\left(\frac{\phi_{0,\eta\eta}}{\phi_0}+\left(\frac{2m-1}{1-m}\right)\frac{\phi_{0,\eta}^2}{\phi_0^2}\right)
+e^{-\gamma\tau}\left(\frac{n-2-m(n+2)}{1-m}\right)\frac{\phi_{0,\eta}}{\phi_0}\right\}.
\end{equation}
This suggests one to consider subsolutions and supersolutions of \eqref{l0-eqn0} of the form
\begin{equation}\label{psi1-defn}
\psi_1(\eta,\tau)=\phi_0(\eta)+e^{-2\gamma\tau}(\phi_1(\eta)+\theta_1\phi_2(\eta))+e^{-\gamma\tau}\theta_2\phi_3(\eta)
\end{equation}
where $\theta_1,\theta_2\in\mathbb{R}$ are constants and $\phi_1$, $\phi_2$ and $\phi_3$ are functions on $(A,\infty)$ which satisfies
\begin{equation}\label{phi12-defn}
\gamma\eta\phi_{i,\eta}(\eta)+(1+2\gamma)\phi_i(\eta)=f_i(\eta)\quad\forall \eta>A, i=1,2
\end{equation}
and 
\begin{equation}\label{phi3-defn}
\gamma\eta\phi_{3,\eta}(\eta)+(1+\gamma)\phi_3(\eta)=f_3(\eta)\quad\forall \eta>A
\end{equation}
respectively with
\begin{equation}\label{f123=}
\left\{\begin{aligned}
&f_1(\eta)=-(n-1)\frac{\phi_{0,\eta\eta}}{\phi_0}=(n-1)\frac{\gamma +1}{\gamma^2}\cdot\frac{A^{\frac{1}{\gamma}}\eta^{-\frac{1}{\gamma}-2}}{1-A^{\frac{1}{\gamma}}\eta^{-\frac{1}{\gamma}}}\\
&f_2(\eta)=-(n-1)\frac{\phi_{0,\eta}^2}{\phi_0^2}=-\frac{(n-1)}{\gamma^2}\cdot\frac{A^{\frac{2}{\gamma}}\eta^{-\frac{2}{\gamma}-2}}{\left(1-A^{\frac{1}{\gamma}}
\eta^{-\frac{1}{\gamma}}\right)^2}\\
&f_3(\eta)=-(n-1)\frac{\phi_{0,\eta}}{\phi_0}=-\frac{(n-1)}{\gamma}\cdot\frac{A^{\frac{1}{\gamma}}\eta^{-\frac{1}{\gamma}-1}}{1-A^{\frac{1}{\gamma}}
\eta^{-\frac{1}{\gamma}}}.\end{aligned}\right.
\end{equation}
Let $\eta_0>A$. By \eqref{phi12-defn}, \eqref{phi3-defn} and \eqref{f123=}, $\forall\eta>A$,
\begin{equation}\label{phi1-2-3-representation}
\left\{\begin{aligned}
&\phi_1(\eta)=\frac{C_1}{\eta^{2+\frac{1}{\gamma}}}+\frac{1}{\gamma\eta^{2+\frac{1}{\gamma}}}\int_{\eta_0}^{\eta}\rho^{1+\frac{1}{\gamma}}f_1(\rho)\,d\rho
=\frac{C_1}{\eta^{2+\frac{1}{\gamma}}}+(n-1)\frac{(\gamma +1)A^{\frac{1}{\gamma}}}{\gamma^3\eta^{2+\frac{1}{\gamma}}}\int_{\eta_0}^{\eta}
\frac{\rho^{-1}}{1-A^{\frac{1}{\gamma}}\rho^{-\frac{1}{\gamma}}}\,d\rho\\
&\phi_2(\eta)=\frac{C_2}{\eta^{2+\frac{1}{\gamma}}}+\frac{1}{\gamma\eta^{2+\frac{1}{\gamma}}}\int_{\eta_0}^{\eta}\rho^{1+\frac{1}{\gamma}}f_2(\rho)\,d\rho
=\frac{C_2}{\eta^{2+\frac{1}{\gamma}}}-\frac{(n-1)A^{\frac{2}{\gamma}}}{\gamma^3\eta^{2+\frac{1}{\gamma}}}\int_{\eta_0}^{\eta}
\frac{\rho^{-1-\frac{1}{\gamma}}}{\left(1-A^{\frac{1}{\gamma}}\rho^{-\frac{1}{\gamma}}\right)^2}\,d\rho\\
&\phi_3(\eta)=\frac{C_3}{\eta^{1+\frac{1}{\gamma}}}+\frac{1}{\gamma\eta^{1+\frac{1}{\gamma}}}\int_{\eta_0}^{\eta}\rho^{\frac{1}{\gamma}}f_3(\rho)\,d\rho
=\frac{C_3}{\eta^{1+\frac{1}{\gamma}}}-\frac{(n-1)A^{\frac{1}{\gamma}}}{\gamma^2\eta^{1+\frac{1}{\gamma}}}\int_{\eta_0}^{\eta}
\frac{\rho^{-1}}{1-A^{\frac{1}{\gamma}}\rho^{-\frac{1}{\gamma}}}\,d\rho\end{aligned}\right.
\end{equation}
for any $\eta>A$ where $C_1, C_2, C_3\in\mathbb{R}$ are constants. As observed in \cite{CDK} by choosing
\begin{equation*}
C_2=\frac{(n-1)A^{\frac{2}{\gamma}}}{\gamma^3}\int_{\eta_0}^{\infty}
\frac{\rho^{-1-\frac{1}{\gamma}}}{\left(1-A^{\frac{1}{\gamma}}\rho^{-\frac{1}{\gamma}}\right)^2}\,d\rho,
\end{equation*}
we get
\begin{equation}\label{phi2-eqn}
\phi_2(\eta)=\frac{(n-1)A^{\frac{2}{\gamma}}}{\gamma^3\eta^{2+\frac{1}{\gamma}}}\int_{\eta}^{\infty}
\frac{\rho^{-1-\frac{1}{\gamma}}}{\left(1-A^{\frac{1}{\gamma}}\rho^{-\frac{1}{\gamma}}\right)^2}\,d\rho>0\quad\forall \eta>A.
\end{equation}
Now by \eqref{phi12-defn} and \eqref{phi3-defn},
\begin{align*}
&\psi_{1,\tau}-\left(\gamma\eta\psi_{1,\eta}+\psi_1-\frac{2(n-1)(n-2-nm)}{1-m}\right)\notag\\\\
=&-2\gamma e^{-2\gamma\tau}(\phi_1+\theta_1\phi_2))-\gamma e^{-\gamma\tau}\theta_2\phi_3
- e^{-2\gamma\tau}[\gamma\eta\phi_{1,\eta}+\phi_1+\theta_1(\gamma\eta\phi_{2,\eta}+\phi_2)]\notag\\
&\qquad - e^{-\gamma\tau}\theta_2(\gamma\eta\phi_{3,\eta}+\phi_3)\notag\\
=&-e^{-2\gamma\tau}(f_1+\theta_1f_2)-e^{-\gamma\tau}\theta_2f_3.
\end{align*}
Hence
\begin{equation}\label{L0-psi1-eqn}
L_0(\psi_1)=(n-1)\left(e^{-2\gamma\tau}I_1+e^{-\gamma\tau}I_2\right)
\end{equation}
where
\begin{equation*}
I_1=\left(\frac{\phi_{0,\eta\eta}}{\phi_0}+\theta_1\frac{\phi_{0,\eta}^2}{\phi_0^2}\right)
-\left(\frac{\psi_{1,\eta\eta}}{\psi_1}+\left(\frac{2m-1}{1-m}\right)\frac{\psi_{1,\eta}^2}{\psi_1^2}
\right)
\end{equation*}
and
\begin{equation*}
I_2=\theta_2\frac{\phi_{0,\eta}}{\phi_0}-\left(\frac{n-2-m(n+2)}{1-m}\right)\frac{\psi_{1,\eta}}{\psi_1}.
\end{equation*}
Let
\begin{equation}\label{h-defn}
h(\eta)=\phi_1(\eta)+\theta_1\phi_2(\eta)\quad\forall \eta>A.
\end{equation}
We now recall some results from \cite{CDK}:

\begin{lem}\label{h-decay-near-a-lem}(cf. Lemma 4.1 of \cite{CDK})
As $\eta\searrow A$, the following holds:
\begin{equation*}
\left\{\begin{aligned}
&h(\eta)=\frac{\theta_1}{\gamma A}\cdot\frac{n-1}{\eta -A}+o((\eta -A)^{-1})\\
&h_{\eta}(\eta)=-\frac{\theta_1}{\gamma A}\cdot\frac{n-1}{(\eta -A)^2}+o((\eta -A)^{-2})\\
&h_{\eta\eta}(\eta)=\frac{2\theta_1}{\gamma A}\cdot\frac{n-1}{(\eta -A)^3}+o((\eta -A)^{-3}).
\end{aligned}\right.
\end{equation*}
\end{lem}

\begin{lem}\label{h-decay-infty-lem}(cf. Lemma 4.2 and Lemma 4.3 of \cite{CDK})
As $\eta\to\infty$, the following holds:
\begin{equation*}
\left\{\begin{aligned}
h(\eta)=&\frac{(n-1)(1+\gamma)}{\gamma^3}A^{\frac{1}{\gamma}}\eta^{-\frac{1}{\gamma}-2}log\,\eta +C_4\eta^{-\frac{1}{\gamma}-2}-\frac{(n-1)(1+\theta_1)}{\gamma^2}A^{\frac{2}{\gamma}}\eta^{-\frac{2}{\gamma}-2}+o\left(\eta^{-\frac{2}{\gamma}-2}\right)\\
h_{\eta}(\eta)=&-\frac{(n-1)(1+\gamma)(1+2\gamma)}{\gamma^4}A^{\frac{1}{\gamma}}\eta^{-\frac{1}{\gamma}-3}log\,\eta +C_5\eta^{-\frac{1}{\gamma}-3}+o\left(\eta^{-\frac{2}{\gamma}-2}\right)\\
h_{\eta\eta}(\eta)=&\frac{(n-1)(1+\gamma)(1+2\gamma)(1+3\gamma)}{\gamma^5}A^{\frac{1}{\gamma}}\eta^{-\frac{1}{\gamma}-4}log\,\eta +C_6\eta^{-\frac{1}{\gamma}-4}+o\left(\eta^{-\frac{2}{\gamma}-3}\right)
\end{aligned}\right.
\end{equation*}
where $C_4, C_5, C_6\in\mathbb{R}$ are some constants.
\end{lem}

\begin{lem}\label{phis-near-A-expansion-lem}
As $\eta\searrow A$, the following holds:

\begin{enumerate}[(i)]
\item
\begin{equation*}
\phi_3(\eta)=\frac{(n-1)}{\gamma A}\log\,\left(\frac{1}{\eta -A}\right)+o(\log\, (\eta-A))
\end{equation*}

\item
\begin{equation*}
\phi_{3,\eta}(\eta)=-\frac{n-1}{\gamma A}\cdot\frac{1}{(\eta -A)}+o((\eta -A)^{-1})
\end{equation*}

\item
\begin{equation*}
\phi_{3,\eta\eta}(\eta)=\frac{n-1}{\gamma A}\cdot\frac{1}{(\eta -A)^2}+o((\eta -A)^{-2}).
\end{equation*}

\end{enumerate}
\end{lem}
\begin{proof}
By \eqref{phi1-2-3-representation} and the l'Hosiptal rule,
\begin{equation*}
\lim_{\eta\searrow A}\frac{\phi_3(\eta)}{\log\,(\eta -A)}=-\frac{(n-1)}{\gamma^2A}
\lim_{\eta\searrow A}\frac{\eta^{-1}(\eta-A)}{1-A^{\frac{1}{\gamma}}\eta^{-\frac{1}{\gamma}}}
=\frac{(n-1)}{\gamma^2A^2}
\lim_{\eta\searrow A}\frac{\eta-A}{A^{\frac{1}{\gamma}}\eta^{-\frac{1}{\gamma}}-1}
=-\frac{(n-1)}{\gamma A}
\end{equation*}
and (i) follows.
By \eqref{phi1-2-3-representation},
\begin{equation}\label{phi3'-repesentation}
\phi_{3,\eta}(\eta)=-\frac{(1+\gamma^{-1})C_3}{\eta^{2+\frac{1}{\gamma}}}
-\frac{(n-1)A^{\frac{1}{\gamma}}}{\gamma^2\eta^{2+\frac{1}{\gamma}}\left(1-A^{\frac{1}{\gamma}}\eta^{-\frac{1}{\gamma}}\right)}
+\frac{(n-1)(1+\gamma^{-1})A^{\frac{1}{\gamma}}}{\gamma^2\eta^{2+\frac{1}{\gamma}}}\int_{\eta_0}^{\eta}
\frac{\rho^{-1}}{1-A^{\frac{1}{\gamma}}\rho^{-\frac{1}{\gamma}}}\,d\rho.
\end{equation}
Hence
\begin{equation}\label{phi3-derivative-limit}
\lim_{\eta\searrow A}(\eta-A)\phi_{3,\eta}(\eta)=-\frac{(n-1)}{\gamma^2A^2}I_1+\frac{(n-1)(1+\gamma^{-1})}{\gamma^2A^2}I_2
\end{equation}
where
\begin{equation}\label{i1-eqn}
I_1=\lim_{\eta\searrow A}\frac{\eta-A}{1-A^{\frac{1}{\gamma}}\eta^{-\frac{1}{\gamma}}}
=\gamma A
\end{equation}
and
\begin{equation}\label{i2-eqn}
I_2=\lim_{\eta\searrow A}\frac{\int_{\eta_0}^{\eta}\frac{\rho^{-1}}{1-A^{\frac{1}{\gamma}}\rho^{-\frac{1}{\gamma}}}\,d\rho}{(\eta-A)^{-1}}
=-A^{-1}\lim_{\eta\searrow A}\frac{(\eta-A)^2}{1-A^{\frac{1}{\gamma}}\eta^{-\frac{1}{\gamma}}}=0.
\end{equation}
By \eqref{phi3-derivative-limit}, \eqref{i1-eqn} and \eqref{i2-eqn}, 
\begin{equation*}
\lim_{\eta\searrow A}(\eta-A)\phi_{3,\eta}(\eta)=-\frac{(n-1)}{\gamma A}
\end{equation*}
and (ii) follows. Differentiating \eqref{phi3'-repesentation} with respect to $\eta$,
\begin{align}\label{phi3''-repesentation}
\phi_{3,\eta\eta}(\eta)=&\frac{(1+\gamma^{-1})(2+\gamma^{-1})C_3}{\eta^{3+\frac{1}{\gamma}}}
+\frac{(n-1)(3+2\gamma^{-1})A^{\frac{1}{\gamma}}}{\gamma^2\eta^{3+\frac{1}{\gamma}}\left(1-A^{\frac{1}{\gamma}}\eta^{-\frac{1}{\gamma}}\right)}
+\frac{(n-1)A^{\frac{2}{\gamma}}}{\gamma^3\eta^{3+\frac{2}{\gamma}}\left(1-A^{\frac{1}{\gamma}}\eta^{-\frac{1}{\gamma}}\right)^2}\notag\\
&\qquad -\frac{(n-1)(1+\gamma^{-1})(2+\gamma^{-1})A^{\frac{1}{\gamma}}}{\gamma^2\eta^{3+\frac{1}{\gamma}}}\int_{\eta_0}^{\eta}
\frac{\rho^{-1}}{1-A^{\frac{1}{\gamma}}\rho^{-\frac{1}{\gamma}}}\,d\rho.
\end{align}
Hence by \eqref{i1-eqn}, \eqref{i2-eqn} and \eqref{phi3''-repesentation},
\begin{equation*}
\lim_{\eta\searrow A}(\eta-A)^2\phi_{3,\eta\eta}(\eta)=\frac{n-1}{\gamma A}
\end{equation*}
and (iii) follows.

\end{proof}

\begin{lem}\label{phi3-decay-at-infty-lem}
As $\eta\to\infty$, the following holds:

\begin{enumerate}[(i)]
\item
\begin{equation*}
\phi_3=\frac{(n-1)}{\gamma^2}\left(-A^{\frac{1}{\gamma}}\eta^{-\frac{1}{\gamma}-1}\log\,\eta+C_7\eta^{-\frac{1}{\gamma}-1}+A^{\frac{2}{\gamma}}\gamma\eta^{-\frac{2}{\gamma}-1}+o(\eta^{-\frac{2}{\gamma}-1})\right)
\end{equation*}

\item

\begin{equation*}
\phi_{3,\eta}=\frac{(n-1)}{\gamma^2}\left(A^{\frac{1}{\gamma}}(1+\gamma^{-1})\eta^{-\frac{1}{\gamma}-2}\log\,\eta+C_8\eta^{-\frac{1}{\gamma}-2}-A^{\frac{2}{\gamma}}(2+\gamma)\eta^{-\frac{2}{\gamma}-2}+o(\eta^{-\frac{2}{\gamma}-2})\right)
\end{equation*}

\item 

\begin{align*}
\phi_{3,\eta\eta}=&\frac{(n-1)}{\gamma^2}\left(-A^{\frac{1}{\gamma}}(1+\gamma^{-1})(2+\gamma^{-1})\eta^{-\frac{1}{\gamma}-3}\log\,\eta+C_9\eta^{-\frac{1}{\gamma}-3}
\right.\\
&\qquad +\left.2A^{\frac{2}{\gamma}}(2+\gamma)(1+\gamma^{-1})\eta^{-\frac{2}{\gamma}-3}+o(\eta^{-\frac{2}{\gamma}-3})\right)
\end{align*}
where $C_7, C_8, C_9\in\mathbb{R}$ are constants.

\end{enumerate}
\end{lem}
\begin{proof}
Since by the Taylor theorem,
\begin{equation}\label{expansion1}
(1-A^{\frac{1}{\gamma}}\rho^{-\frac{1}{\gamma}})^{-1}=1+A^{\frac{1}{\gamma}}\rho^{-\frac{1}{\gamma}}
+A^{\frac{2}{\gamma}}\rho^{-\frac{2}{\gamma}}+o(\rho^{-\frac{2}{\gamma}})\quad\mbox{ as }\rho\to\infty,
\end{equation}
we have
\begin{align}\label{integral1}
\int_{\eta_0}^{\eta}\frac{\rho^{-1}}{1-A^{\frac{1}{\gamma}}\rho^{-\frac{1}{\gamma}}}\,d\rho
=&\int_{\eta_0}^{\eta}\left(\rho^{-1}+A^{\frac{1}{\gamma}}\rho^{-\frac{1}{\gamma}-1}+A^{\frac{2}{\gamma}}\rho^{-\frac{2}{\gamma}-1}+o(\rho^{-\frac{2}{\gamma}-1})\right)\,d\rho\notag\\
=&\log\,\eta -\gamma A^{\frac{1}{\gamma}}\eta^{-\frac{1}{\gamma}}-\frac{\gamma A^{\frac{2}{\gamma}}}{2}\eta^{-\frac{2}{\gamma}}+o(\eta^{-\frac{2}{\gamma}})+C\quad\mbox{ as }\eta\to\infty
\end{align}
where $C=C(\eta_0)$ is some constant. By \eqref{phi1-2-3-representation} and \eqref{integral1}, (i) follows.
By \eqref{phi3'-repesentation}, \eqref{expansion1} and \eqref{integral1}, (ii) follows.
By \eqref{phi3''-repesentation}, \eqref{expansion1} and \eqref{integral1},  (iii) follows.
\end{proof}

Note that by Lemma \ref{h-decay-near-a-lem}, Lemma \ref{phis-near-A-expansion-lem}, \eqref{phi0-defn} and the Taylor theorem, there exist constants $0<\delta_1<1$ and $\kappa_1>1>\kappa_2>0$ such that
\begin{equation}\label{h-near-A-expression}
|h(\eta)(\eta -A)|,\quad |h_{\eta}(\eta)(\eta -A)^2|,\quad |h_{\eta\eta}(\eta)(\eta -A)^3|<\kappa_1|\theta_1|\quad\forall A<\eta\le A+\delta_1,
\end{equation}
\begin{equation}\label{phi3-near-A-expression}
\left\{\begin{aligned}
&\kappa_2\log\,\left(\frac{1}{\eta -A}\right)\le \phi_3(\eta)\le \kappa_1\log\,\left(\frac{1}{\eta -A}\right)\quad\forall A<\eta\le A+\delta_1\\
&-\frac{\kappa_2}{\eta -A}\ge \phi_{3,\eta}(\eta)\ge-\frac{\kappa_1}{\eta -A}\qquad\qquad\qquad\forall A<\eta\le A+\delta_1\\
&\frac{\kappa_2}{(\eta -A)^2}\le \phi_{3,\eta\eta}(\eta)\le\frac{\kappa_1}{(\eta -A)^2}\qquad\qquad\quad\forall A<\eta\le A+\delta_1,
\end{aligned}\right.
\end{equation}
\begin{equation}\label{phi0-2-der-estimates}
\left\{\begin{aligned}
&\left|\phi_0(\eta)-\frac{a_0}{\gamma A}(\eta-A)\right|<\kappa_1(\eta -A)^2\quad\forall A<\eta\le A+\delta_1\\
&\left|\phi_{0,\eta}(\eta)-\frac{a_0}{\gamma A}\right|<\kappa_1(\eta -A)\qquad\quad\,\,\,\forall A<\eta\le A+\delta_1\\
&|\phi_{0,\eta\eta}(\eta)|<\kappa_1\qquad\qquad\qquad\qquad\quad\forall A<\eta\le A+\delta_1
\end{aligned}\right.
\end{equation}
and
\begin{equation}\label{phi0-positive}
0<\frac{3a_0}{4\gamma A}(\eta-A)\le\phi_0(\eta)\le\frac{3a_0}{2\gamma A}(\eta-A)\quad\forall A<\eta\le A+\delta_1.
\end{equation}

\begin{lem}\label{psi1-positive-lem}
Let $\gamma>0$, $\theta_1\in\mathbb{R}$, $\theta_2\ge 0$ and $\psi_1$ be given by \eqref{psi1-defn}.  Then there exist constants $\xi_1>0$ and $\tau_1\ge 0$ such that 
\begin{equation}\label{psi1-positive10}
\psi_1(\eta,\tau)>0\quad\forall\eta\ge A+\xi_1e^{-\gamma\tau}, \tau\ge\tau_1.
\end{equation}
\end{lem}
\begin{proof}
Let $a_0$ be given by \eqref{a0-defn} and 
\begin{equation}\label{xi1-defn}
\xi_1=\sqrt{\frac{4\gamma \kappa_1 A|\theta_1|}{a_0}}.
\end{equation}
By  Lemma \ref{h-decay-infty-lem} and Lemma \ref{phi3-decay-at-infty-lem}, there exists  a constant $c_0>0$ such that 
\begin{equation}\label{h-phi3-bd}
|h(\eta)|\le c_0,\quad |\phi_3(\eta)|\le c_0\quad\forall \eta\ge A+\delta_1.
\end{equation}
On the other hand  by \eqref{phi0-defn},
\begin{equation}\label{phi0-positive-at-infty}
\phi_0(\eta)\ge\phi_0(A+\delta_1)\quad\forall \eta\ge A+\delta_1.
\end{equation}
Let 
\begin{equation*}
\tau_1=\frac{2}{\gamma}max\left(0,\log\,\left(\frac{3(1+\theta_2)c_0}{\phi_0(A+\delta_1)}\right),\log \,\left(\frac{\xi_1}{\delta_1}\right)\right).
\end{equation*}
Since $\theta_2\ge 0$, by \eqref{psi1-defn},  \eqref{h-near-A-expression}, \eqref{phi3-near-A-expression}, \eqref{phi0-positive} and \eqref{xi1-defn},
\begin{align}
\psi_1(\eta)\ge&\phi_0(\eta)+e^{-2\gamma\tau}h(\eta)\qquad\qquad\,\,\,\forall A<\eta\le A+\delta_1, \tau\ge\tau_1\label{psi1-lower-bd0}\\
\ge&\frac{3a_0}{4\gamma A}(\eta-A)-\frac{\kappa_1|\theta_1|e^{-2\gamma\tau}}{\eta -A}\quad\forall A<\eta\le A+\delta_1, \tau\ge\tau_1\notag\\
\ge&\left(\frac{3a_0}{4\gamma A}-\frac{\kappa_1|\theta_1|}{\xi_1^2}\right)(\eta-A)\qquad\forall \xi_1e^{-\gamma\tau}\le\eta -A\le\delta_1,\tau\ge\tau_1\notag\\
\ge&\frac{a_0}{2\gamma A}(\eta-A)>0\qquad\qquad\quad\forall \xi_1e^{-\gamma\tau}\le\eta -A\le\delta_1, \tau\ge\tau_1.\label{psi1-positive11}
\end{align}
By \eqref{psi1-defn}, \eqref{h-phi3-bd} and \eqref{phi0-positive-at-infty},
\begin{equation}\label{psi1-positive12}
\psi_1(\eta,\tau)\ge\frac{1}{3}\phi_0(A+\delta_1)>0\quad\forall \eta\ge A+\delta_1,\tau\ge\tau_1.
\end{equation}
By \eqref{psi1-positive11} and \eqref{psi1-positive12}, we get \eqref{psi1-positive10} and the lemma follows. 

\end{proof}

\begin{lem}\label{sub-supersoln-outer-region-near-left-bdary-lem}
Let $\gamma>0$, $\theta_1\in\mathbb{R}$, $\theta_2\ge 0$ and $\psi_1$ be given by \eqref{psi1-defn}.  Let $\xi_1>0$ and $\tau_1\ge 0$ be as in  Lemma \ref{psi1-positive-lem}.  Then there exist constants $\xi_0\ge\xi_1$,  $0<\delta_0<\delta_1$ and $\tau_2>\tau_1$ such that 
 $\psi_1$  is a subsolution of \eqref{l0-eqn0} if
\begin{equation}\label{theta1-2-ineqn11}
\theta_1<\frac{2m-1}{1-m},\quad \theta_2=0\quad\mbox{ and }\quad 0<m<\frac{n-2}{n+2}
\end{equation}
and a supersolution of \eqref{l0-eqn0} if
\begin{equation}\label{theta1-2-ineqn12}
\theta_1>\max\left(0,\frac{2m-1}{1-m}\right)\quad\mbox{ and }\quad\theta_2>\max\left(0,\frac{n-2-m(n+2)}{1-m}\right)
\end{equation}
in the region
\begin{equation*}
\{(\eta,\tau):\xi_0e^{-\gamma\tau}\le\eta -A\le\delta_0, \tau\ge\tau_2\}.
\end{equation*}
\end{lem}
\begin{proof}
Let $\xi_0\ge\xi_1$,  $0<\delta_0<\delta_1$, $0<\varepsilon<1$  and 
\begin{equation}\label{tau2-defn}
\tau_2>max\left(\tau_1,\frac{1}{\gamma}log \,\left(\frac{\xi_0}{\delta_0}\right)\right)
\end{equation}
be constants to be determined later. 
By the proof of Lemma  \ref{psi1-positive-lem}, \eqref{psi1-lower-bd0} and \eqref{psi1-positive11} holds.
By \eqref{phi0-defn} and \eqref{phi0-positive},
\begin{equation}\label{phi0-phi0'-ration-lower-bd10}
\frac{\phi_{0,\eta}}{\phi_0}\ge\frac{\frac{a_0}{\gamma}A^{\frac{1}{\gamma}}\eta^{-\frac{1}{\gamma}-1}}{\frac{3a_0}{2\gamma A}(\eta -A)}=\frac{2}{3}\left(\frac{A}{A+\delta_1}\right)^{\frac{1}{\gamma}+1}\frac{1}{\eta -A}\quad\forall A<\eta\le A+\delta_1.
\end{equation}
Similarly there exists a constant $C_1>0$ such that
\begin{equation}\label{phi0-phi0'-square-ration-lower-bd10}
\frac{\phi_{0,\eta}^2}{\phi_0}\ge\frac{C_1}{\eta -A}\quad\forall A<\eta\le A+\delta_1.
\end{equation}
By \eqref{psi1-defn}, \eqref{h-near-A-expression}, \eqref{phi3-near-A-expression} and \eqref{phi0-2-der-estimates},
\begin{equation}\label{psi''-bd}
\left|\psi_{1,\eta\eta}\right|\le\kappa_1+\frac{\kappa_1|\theta_1|}{(\eta -A)^3}e^{-2\gamma\tau}+\frac{\kappa_1\theta_2}{(\eta -A)^2}e^{-\gamma\tau}
\le\left(\kappa_1\delta_0+\frac{\kappa_1|\theta_1|}{\xi_0^2}+\frac{\kappa_1\theta_2}{\xi_0}\right)\frac{1}{\eta-A}
\end{equation}
for any $\xi_0e^{-\gamma\tau}\le \eta -A\le\delta_0$, $\tau\ge\tau_2$. Let $0<\varepsilon_1<C_1\varepsilon/4$.
By  \eqref{psi''-bd} we can choose $\delta_0>0$ sufficiently small and $\xi_0$, $\tau_2$, sufficiently large such that \eqref{tau2-defn} holds and
\begin{equation}\label{psi''-bd5}
\left|\psi_{1,\eta\eta}\right|\le\frac{\varepsilon_1}{\eta-A}\quad\forall \xi_0e^{-\gamma\tau}\le\eta -A\le\delta_0, \tau\ge\tau_2.
\end{equation}
By \eqref{psi1-defn}, \eqref{h-near-A-expression} and \eqref{phi3-near-A-expression},
\begin{align}\label{psi'-phi0'-ratio-upper-bd}
\frac{\psi_{1,\eta}}{\phi_{0,\eta}}
\le&\frac{\frac{a_0A^{\frac{1}{\gamma}}}{\gamma}\eta^{-\frac{1}{\gamma}-1}+\frac{\kappa_1|\theta_1|e^{-2\gamma\tau}}{(\eta -A)^2}}{\frac{a_0A^{\frac{1}{\gamma}}}{\gamma}\eta^{-\frac{1}{\gamma}-1}}
\le 1+\frac{\gamma\kappa_1|\theta_1|(A+\delta_0)^{\frac{1}{\gamma}+1}}{a_0A^{\frac{1}{\gamma}}\xi_0^2}
\quad\forall \xi_0e^{-\gamma\tau}\le \eta -A\le\delta_0, \tau\ge\tau_2.
\end{align}
Similarly,
\begin{align}\label{psi'-phi0'-ratio-upper-bd2}
\frac{\psi_{1,\eta}}{\phi_{0,\eta}}
\ge&\frac{\frac{a_0A^{\frac{1}{\gamma}}}{\gamma}\eta^{-\frac{1}{\gamma}-1}-\frac{\kappa_1|\theta_1|e^{-2\gamma\tau}}{(\eta -A)^2}-\frac{\kappa_1\theta_2e^{-\gamma\tau}}{\eta -A}}{\frac{a_0A^{\frac{1}{\gamma}}}{\gamma}\eta^{-\frac{1}{\gamma}-1}}
\ge 1-\left(\frac{\kappa_1|\theta_1|}{\xi_0^2}+\frac{\kappa_1\theta_2}{\xi_0}\right)\frac{\gamma(A+\delta_0)^{\frac{1}{\gamma}+1}}{a_0A^{\frac{1}{\gamma}}}
\end{align}
holds for any $\xi_0e^{-\gamma\tau}\le \eta -A\le\delta_0$, $\tau\ge\tau_2$. 
By \eqref{h-near-A-expression},  \eqref{phi0-positive} and \eqref{psi1-lower-bd0},
\begin{equation}\label{psi1>phi0-eqn1}
\psi_1(\eta)\ge\phi_0(\eta)\left(1-\frac{\frac{\kappa_1|\theta_1|}{\eta -A}e^{-2\gamma\tau}}{\frac{3a_0}{4\gamma A}(\eta-A)}\right)
\ge\phi_0(\eta)\left(1-\frac{4\gamma\kappa_1A|\theta_1|}{3a_0\xi_0^2}\right)
\quad\forall \xi_0e^{-\gamma\tau}\le \eta -A\le\delta_0, \tau\ge\tau_2.
\end{equation}
By \eqref{psi''-bd5}, \eqref{psi'-phi0'-ratio-upper-bd}, \eqref{psi'-phi0'-ratio-upper-bd2} and \eqref{psi1>phi0-eqn1}, we can choose $\delta_0>0$ sufficiently small and $\xi_0$, $\tau_2$, sufficiently large such that \eqref{tau2-defn} holds and
\begin{equation}\label{ratio-upper-bd1}
\left\{\begin{aligned}
&0<(1-\varepsilon)\frac{\phi_{0,\eta}}{\psi_1}\le\frac{\psi_{1,\eta}}{\psi_1}\le (1+\varepsilon)\frac{\phi_{0,\eta}}{\phi_0}\quad\forall \xi_0e^{-\gamma\tau}\le \eta -A\le\delta_0,\tau\ge\tau_2\\
&\frac{\psi_{1,\eta}^2}{\psi_1^2}\le (1+\varepsilon)\frac{\phi_{0,\eta}^2}{\phi_0^2}\qquad\qquad\qquad\qquad\forall \xi_0e^{-\gamma\tau}\le \eta -A\le\delta_0,\tau\ge\tau_2\end{aligned}\right.
\end{equation}
and
\begin{equation}\label{psi''-bd7}
\left|\frac{\psi_{1,\eta\eta}}{\psi_1}\right|\le\frac{2\varepsilon_1}{(\eta-A)\phi_0}<\frac{C_1\varepsilon}{2(\eta-A)\phi_0}\quad\forall \xi_0e^{-\gamma\tau}\le\eta -A\le\delta_0, \tau\ge\tau_2.
\end{equation}
Note that in fact we have
\begin{equation*}
\frac{\phi_{0,\eta}}{\psi_1}\to\frac{\phi_{0,\eta}}{\phi_0}\quad\mbox{ as }\quad 
(\eta -A)e^{\gamma\tau}\to\infty, \eta\to A^+,\tau\to\infty.
\end{equation*}
By \eqref{phi0-phi0'-square-ration-lower-bd10} and \eqref{psi''-bd7},
\begin{equation}\label{phi1-phi0-derivative-ratio-bd15}
\left|\frac{\psi_{1,\eta\eta}}{\psi_1}\right|\le\varepsilon\frac{\phi_{0,\eta}^2}{\phi_0^2}\quad\forall \xi_0e^{-\gamma\tau}\le\eta -A\le\delta_0, \tau\ge\tau_2.
\end{equation}
Similarly by choosing $\delta_0>0$ sufficiently small and $\xi_0$, $\tau_2$, sufficiently large such that \eqref{tau2-defn} holds we have 
\begin{equation}\label{phi0-phi0''-derivative-ratio-bd15}
\left|\frac{\phi_{0,\eta\eta}}{\phi_0}\right|\le\varepsilon\frac{\phi_{0,\eta}^2}{\phi_0^2}\quad\forall A<\eta
\le A+\delta_0.
\end{equation}
We now suppose either \eqref{theta1-2-ineqn11} or \eqref{theta1-2-ineqn12} holds and divide the proof into two cases:

\noindent\textbf{Case 1}: \eqref{theta1-2-ineqn11} holds.

\noindent Since by \eqref{theta1-2-ineqn11} $\theta_2=0$, by \eqref{psi1-defn}, \eqref{h-near-A-expression} and  \eqref{phi0-positive},
\begin{equation}\label{psi1<phi0-eqn1}
\psi_1(\eta)\le\phi_0(\eta)\left(1+\frac{\frac{\kappa_1|\theta_1|}{\eta -A}e^{-2\gamma\tau}}{\frac{3a_0}{4\gamma A}(\eta-A)}\right)
\le\phi_0(\eta)\left(1+\frac{4\gamma\kappa_1A|\theta_1|}{3a_0\xi_0^2}\right)
\quad\forall \xi_0e^{-\gamma\tau}\le \eta -A\le\delta_0, \tau\ge\tau_2.
\end{equation}
By \eqref{psi'-phi0'-ratio-upper-bd2} and \eqref{psi1<phi0-eqn1},
we can choose $\delta_0>0$ sufficiently small and $\xi_0$, $\tau_2$, sufficiently large such that \eqref{tau2-defn} holds and
\begin{equation}\label{ratio-upper-bd21}
\frac{\psi_{1,\eta}^2}{\psi_1^2}\ge (1-\varepsilon)\frac{\phi_{0,\eta}^2}{\phi_0^2}\quad\forall \xi_0e^{-\gamma\tau}\le \eta -A\le\delta_0,\tau\ge\tau_2.
\end{equation}
By \eqref{ratio-upper-bd1}, \eqref{phi1-phi0-derivative-ratio-bd15}, \eqref{phi0-phi0''-derivative-ratio-bd15} and \eqref{ratio-upper-bd21},
\begin{align}\label{I1-upper-bd}
I_1=&\left(\theta_1-\left(\frac{2m-1}{1-m}\right)\right)\frac{\phi_{0,\eta}^2}{\phi_0^2}
+\left(\frac{2m-1}{1-m}\right)\left(\frac{\phi_{0,\eta}^2}{\phi_0^2}-\frac{\psi_{1,\eta}^2}{\psi_1^2}
\right)
+\frac{\phi_{0,\eta\eta}}{\phi_0}-\frac{\psi_{1,\eta\eta}}{\psi_1}
\notag\\
\le&\left\{\theta_1-\left(\frac{2m-1}{1-m}\right)+\varepsilon\left|\frac{2m-1}{1-m}\right|+2\varepsilon\right\}\frac{\phi_{0,\eta}^2}{\phi_0^2}
\end{align}
and
\begin{equation}\label{I2-upper-bd}
I_2\le-\frac{(n-2-m(n+2))}{1-m}\cdot\frac{\psi_{1,\eta}}{\psi_1}\le 0
\end{equation}
for any $\xi_0e^{-\gamma\tau}\le \eta -A\le\delta_0$, $\tau\ge\tau_2$.
We now choose
\begin{equation*}
0<\varepsilon<\frac{\frac{2m-1}{1-m}-\theta_1}{2+\left|\frac{2m-1}{1-m}\right|}.
\end{equation*}
Then by \eqref{I1-upper-bd},
\begin{equation}\label{I1-upper-bd2}
I_1\le 0\quad\forall \xi_0e^{-\gamma\tau}\le\eta -A\le\delta_0, \tau\ge\tau_2.
\end{equation}
Hence by \eqref{L0-psi1-eqn}, \eqref{I2-upper-bd} and \eqref{I1-upper-bd2},
\begin{equation*}
L_0(\psi_1)\le 0\quad\forall \delta_2\le \eta -A\le\delta_1, \tau\ge\tau_2.
\end{equation*}
Thus $\psi_1$ is a subsolution of \eqref{l0-eqn0} in the region
$\{(\eta,\tau):\xi_0e^{-\gamma\tau}\le\eta -A\le\delta_0, \tau\ge\tau_2\}$.

\noindent\textbf{Case 2}: \eqref{theta1-2-ineqn12} holds.

\noindent By \eqref{ratio-upper-bd1}, \eqref{phi1-phi0-derivative-ratio-bd15} and \eqref{phi0-phi0''-derivative-ratio-bd15},
\begin{equation}\label{I1-lower-bd}
I_1\ge\left(\theta_1-2\varepsilon-(1+\varepsilon)\max\left(0,\frac{2m-1}{1-m}\right)\right)\frac{\phi_{0,\eta}^2}{\phi_0^2}
\end{equation}
and
\begin{equation}\label{I2-lower-bd}
I_2\ge\left(\theta_2-(1+\varepsilon)\max\left(0,\frac{n-2-m(n+2)}{1-m}\right)\right)\frac{\phi_{0,\eta}}{\phi_0}
\end{equation}
for any $\xi_0e^{-\gamma\tau}\le \eta -A\le\delta_0$, $\tau\ge\tau_2$.
We now choose
\begin{equation*}
0<\varepsilon<\min\left(\frac{\theta_1-\max\left(0,\frac{2m-1}{1-m}\right)}{2+\max\left(0,\frac{2m-1}{1-m}\right)},\frac{\theta_2-\max\left(0,\frac{n-2-m(n+2)}{1-m}\right)}{1+\max\left(0,\frac{n-2-m(n+2)}{1-m}\right)} \right).
\end{equation*}
Then by \eqref{L0-psi1-eqn}, \eqref{I1-lower-bd} and \eqref{I2-lower-bd},
\begin{align*}
&I_1\ge 0\quad\mbox{ and }\quad I_2\ge 0\quad\forall \xi_0e^{-\gamma\tau}\le\eta -A\le\delta_0, \tau\ge\tau_2\notag\\
\Rightarrow\quad&L_0(\psi_1)\ge 0\qquad\qquad\quad\,\,\,\forall \xi_0e^{-\gamma\tau}\le\eta -A\le\delta_0, \tau\ge\tau_2.
\end{align*}
Hence $\psi_1$ is a supersolution of \eqref{l0-eqn0} in the region
$\{(\eta,\tau):\xi_0e^{-\gamma\tau}\le\eta -A\le\delta_0, \tau\ge\tau_2\}$
and the lemma follows.
\end{proof}

\begin{lem}\label{sub-supersoln-outer-far-field-region-lem}
Let $\gamma>1$, $\theta_1\in\mathbb{R}$, $\theta_2\ge 0$ and $\psi_1$ be given by \eqref{psi1-defn}.   Let  $0<\delta_0<\delta_1$  be as in Lemma \ref{sub-supersoln-outer-region-near-left-bdary-lem} and $\tau_1\ge 0$ be as in  Lemma \ref{psi1-positive-lem}. Then there exists  $\tau_3>\tau_1$ such that $\psi_1$  is a subsolution of \eqref{l0-eqn0} if
\begin{equation}\label{theta1-2-ineqn1}
\theta_1<\frac{2m-1}{1-m}\quad\mbox{ and }\quad\theta_2<\frac{n-2-m(n+2)}{1-m}
\end{equation}
and a supersolution of \eqref{l0-eqn0} if
\begin{equation}\label{theta1-2-ineqn2}
\theta_1>\frac{2m-1}{1-m}\quad\mbox{ and }\quad\theta_2>\frac{n-2-m(n+2)}{1-m}
\end{equation}
holds in the region
\begin{equation*}
\{(\eta,\tau):\eta\ge A+\delta_0, \tau\ge\tau_3\}.
\end{equation*}
\end{lem}
\begin{proof}
We will use a modification of the proof of claim 4.6 of \cite{CDK} to prove the proposition. Suppose either \eqref{theta1-2-ineqn1} or \eqref{theta1-2-ineqn2} holds.
By \eqref{phi0-defn}, Lemma \ref{h-decay-infty-lem} and Lemma \ref{phi3-decay-at-infty-lem}, there exist constants $\kappa_3>e$ and $M_1>0$ such that
\begin{equation}\label{h-phi0-derivative-ratio-bd}
\left|\frac{h(\eta)}{\phi_0(\eta)}\right|,\quad\left|\frac{h_{\eta}(\eta)}{\phi_{0,\eta}(\eta)}\right|,\quad\left|\frac{h_{\eta\eta}(\eta)}{\phi_{0,\eta\eta}(\eta)}\right|\le\kappa_3\frac{\log\,(3+\eta)}{\eta^2}\le M_1\kappa_3\quad\forall\eta\ge A+\delta_0
\end{equation}
and
\begin{equation}\label{phi3-phi0-derivative-ratio-bd}
\theta_2\left|\frac{\phi_3(\eta)}{\phi_0(\eta)}\right|,\quad\theta_2\left|\frac{\phi_{3,\eta}(\eta)}{\phi_{0,\eta}(\eta)}\right|,\quad\theta_2\left|\frac{\phi_{3,\eta\eta}(\eta)}{\phi_{0,\eta\eta}(\eta)}\right|\le\kappa_3\frac{\log\,(3+\eta)}{\eta}\le M_1\kappa_3\quad\forall\eta\ge A+\delta_0.
\end{equation}
By \eqref{L0-psi1-eqn},
\begin{align}\label{L0-psi1-eqn1}
L_0(\psi_1)=&(n-1)\left\{e^{-2\gamma\tau}\left[\left(\frac{\phi_{0,\eta\eta}}{\phi_0}-\frac{\psi_{1,\eta\eta}}{\psi_1}\right)
+\left(\frac{2m-1}{1-m}\right)\left(\frac{\phi_{0,\eta}^2}{\phi_0^2}-\frac{\psi_{1,\eta}^2}{\psi_1^2}
\right)
+\left(\theta_1-\left(\frac{2m-1}{1-m}\right)\right)\frac{\phi_{0,\eta}^2}{\phi_0^2}\right]\right.\notag\\
&\qquad+e^{-\gamma\tau}\left.\left[\left(\frac{n-2-m(n+2)}{1-m}\right)\left(\frac{\phi_{0,\eta}}{\phi_0}-\frac{\psi_{1,\eta}}{\psi_1}\right)+\left(\theta_2-\left(\frac{n-2-m(n+2)}{1-m}\right)\right)\frac{\phi_{0,\eta}}{\phi_0}\right]\right\}.
\end{align}
Let
\begin{equation}\label{tau3-defn}
\tau_3>\max \left(\tau_1,\frac{1}{\gamma}\log\,(4M_1\kappa_3)\right)
\end{equation}
be a constant to be determined later. Then by \eqref{h-phi0-derivative-ratio-bd}, \eqref{phi3-phi0-derivative-ratio-bd} and \eqref{tau3-defn},
\begin{equation}\label{estimate10}
1-e^{-2\gamma\tau}\left|\frac{h}{\phi_0}\right|-e^{-\gamma\tau}\theta_2\left|\frac{\phi_3}{\phi_0}\right|\ge 1-\left(e^{-2\gamma\tau}+e^{-\gamma\tau}\right)M_1\kappa_3\ge\frac{1}{2}\quad\forall\eta\ge A+\delta_0,\tau\ge\tau_3.
\end{equation}
By direct computation,
\begin{equation*}
\left|\frac{\phi_{0,\eta\eta}}{\phi_0}-\frac{\psi_{1,\eta\eta}}{\psi_1}\right|
\le\frac{e^{-2\gamma\tau}\left|\frac{h}{\phi_0}-\frac{h_{\eta\eta}}{\phi_{0,\eta\eta}}\right|
+e^{-\gamma\tau}\theta_2\left|\frac{\phi_3}{\phi_0}-\frac{\phi_{3,\eta\eta}}{\phi_{0,\eta\eta}}\right|}{1-e^{-2\gamma\tau}\left|\frac{h}{\phi_0}\right|-e^{-\gamma\tau}\theta_2\left|\frac{\phi_3}{\phi_0}\right|}
\cdot\left|\frac{\phi_{0,\eta\eta}}{\phi_0}\right|,
\end{equation*}
\begin{align*}
&\left|\frac{\phi_{0,\eta}^2}{\phi_0^2}-\frac{\psi_{1,\eta}^2}{\psi_1^2}\right|\notag\\
\le&\left\{2e^{-\gamma\tau}\theta_2\left|\frac{\phi_3}{\phi_0}-\frac{\phi_{3,\eta}}{\phi_{0,\eta}}\right|
+2e^{-2\gamma\tau}\left|\frac{h}{\phi_0}-\frac{h_{\eta}}{\phi_{0,\eta}}\right|
+e^{-2\gamma\tau}\theta_2^2\left|\frac{\phi_3^2}{\phi_0^2}-\frac{\phi_{3,\eta}^2}{\phi_{0,\eta}^2}\right|
+2e^{-3\gamma\tau}\theta_2\left|\frac{h\phi_3}{\phi_0^2}-\frac{h_{\eta}\phi_{3,\eta}}{\phi_{0,\eta}^2}\right|\right.\notag\\
&\qquad +\left.e^{-4\gamma\tau}\left|\frac{h^2}{\phi_0^2}-\frac{h_{\eta}^2}{\phi_{0,\eta}^2}\right|
\right\}\cdot
\frac{1}{\left(1-e^{-2\gamma\tau}\left|\frac{h}{\phi_0}\right|
-e^{-\gamma\tau}\theta_2\left|\frac{\phi_3}{\phi_0}\right|\right)^2}\cdot\frac{\phi_{0,\eta}^2}{\phi_0^2}
\end{align*}
and
\begin{equation*}
\left|\frac{\phi_{0,\eta}}{\phi_0}-\frac{\psi_{1,\eta}}{\psi_1}\right|
\le\frac{e^{-2\gamma\tau}\left|\frac{h}{\phi_0}-\frac{h_{\eta}}{\phi_{0,\eta}}\right|
+e^{-\gamma\tau}\theta_2\left|\frac{\phi_3}{\phi_0}-\frac{\phi_{3,\eta}}{\phi_{0,\eta}}\right|}{1-e^{-2\gamma\tau}\left|\frac{h}{\phi_0}\right|-e^{-\gamma\tau}\theta_2\left|\frac{\phi_3}{\phi_0}\right|}\cdot
\frac{\phi_{0,\eta}}{\phi_0}
\end{equation*}
holds for any $\eta\ge A+\delta_0,\tau\ge\tau_3$.
Hence  by \eqref{h-phi0-derivative-ratio-bd}, \eqref{phi3-phi0-derivative-ratio-bd}, \eqref{tau3-defn}, \eqref{estimate10} and by choosing $\tau_3$ sufficiently large we have
\begin{equation}\label{phi0-ratio-bd12}
\left|\frac{\phi_{0,\eta\eta}}{\phi_0}-\frac{\psi_{1,\eta\eta}}{\psi_1}\right|
\le 20e^{-\gamma\tau}\kappa_3\left|\frac{\phi_{0,\eta\eta}}{\phi_0}\right|\frac{\log\,(3+\eta)}{\eta}\quad\forall\eta\ge A+\delta_0,\tau\ge\tau_3,
\end{equation}
\begin{equation}\label{phi0-ratio-bd13}
\left|\frac{\phi_{0,\eta}^2}{\phi_0^2}-\frac{\psi_{1,\eta}^2}{\psi_1^2}\right|
\le 20e^{-\gamma\tau}\kappa_3\frac{\phi_{0,\eta}^2}{\phi_0^2}\frac{\log\,(3+\eta)}{\eta}\quad\forall\eta\ge A+\delta_0,\tau\ge\tau_3
\end{equation}
and
\begin{equation}\label{phi0-ratio-bd14}
\left|\frac{\phi_{0,\eta}}{\phi_0}-\frac{\psi_{1,\eta}}{\psi_1}\right|
\le 20e^{-\gamma\tau}\kappa_3\frac{\phi_{0,\eta}}{\phi_0}\frac{\log\,(3+\eta)}{\eta}\quad\forall\eta\ge A+\delta_0,\tau\ge\tau_3.
\end{equation}
By \eqref{phi0-defn} there exist constants $C_1>0$, $C_2>0$, such that
\begin{equation}\label{phi0''-phi0-ratio-upper-bd}
\left|\frac{\phi_{0,\eta\eta}}{\phi_0}\right|\le C_1\eta^{-\frac{1}{\gamma}-2}\quad\forall \eta\ge A+\delta_0
\end{equation}
and
\begin{equation}\label{phi0'-phi0-ratio-square-lower-bd}
\frac{\phi_{0,\eta}}{\phi_0}\ge C_2\eta^{-\frac{1}{\gamma}-1}\quad\forall \eta\ge A+\delta_0.
\end{equation}
By \eqref{phi0''-phi0-ratio-upper-bd} and \eqref{phi0'-phi0-ratio-square-lower-bd} and the fact that $\gamma>1$ we have
\begin{equation}\label{phi0'-phi0-ratio-upper-lower-bd-compare}
\left|\frac{\phi_{0,\eta\eta}}{\phi_0}\right|\frac{\log\,(3+\eta)}{\eta}\le C_3\frac{\phi_{0,\eta}^2}{\phi_0^2}\quad\forall \eta\ge A+\delta_0
\end{equation}
for some constant $C_3>0$. By \eqref{phi0-ratio-bd12}, \eqref{phi0-ratio-bd13}, \eqref{phi0-ratio-bd14} and \eqref{phi0'-phi0-ratio-upper-lower-bd-compare} and by choosing $\tau_3$ sufficiently large we have
\begin{equation}\label{compare-bd1}
\left|\frac{\phi_{0,\eta\eta}}{\phi_0}-\frac{\psi_{1,\eta\eta}}{\psi_1}\right|
\le\frac{1}{4}\left|\theta_1-\left(\frac{2m-1}{1-m}\right)\right|\frac{\phi_{0,\eta}^2}{\phi_0^2}
\end{equation}
\begin{equation}\label{compare-bd1a}
\left|\frac{2m-1}{1-m}\right|\left|\frac{\phi_{0,\eta}^2}{\phi_0^2}-\frac{\psi_{1,\eta}^2}{\psi_1^2}
\right|
\le\frac{1}{4}\left|\theta_1-\left(\frac{2m-1}{1-m}\right)\right|\frac{\phi_{0,\eta}^2}{\phi_0^2}
\end{equation}
and
\begin{equation}\label{compare-bd2}
\left|\frac{n-2-m(n+2)}{1-m}\right|\left|\frac{\phi_{0,\eta}}{\phi_0}-\frac{\psi_{1,\eta}}{\psi_1}\right|\le\frac{1}{2}\left|\theta_2-\left(\frac{n-2-m(n+2)}{1-m}\right)\right|\frac{\phi_{0,\eta}}{\phi_0}
\end{equation}
hold for any $\eta\ge A+\delta_0,\tau\ge\tau_3$. By \eqref{L0-psi1-eqn1}, \eqref{compare-bd1}, \eqref{compare-bd1a} and \eqref{compare-bd2} we get that $\psi_1$  is a subsolution of \eqref{l0-eqn0} if \eqref{theta1-2-ineqn1} holds or a supersolution of \eqref{l0-eqn0} if \eqref{theta1-2-ineqn2} holds in the region
$\{(\eta,\tau):\eta\ge A+\delta_0, \tau\ge\tau_3\}$ and the lemma follows.
\end{proof}

By Lemma \ref{psi1-positive-lem}, Lemma \ref{sub-supersoln-outer-region-near-left-bdary-lem} and Lemma \ref{sub-supersoln-outer-far-field-region-lem} we have the following result.

\begin{prop}\label{sub-supersoln-outer-region-lem}
Let $\gamma>1$, $\theta_1\in\mathbb{R}$, $\theta_2\ge 0$ and $\psi_1$ be given by \eqref{psi1-defn}.  Let $\xi_1>0$ and $\tau_1\ge 0$ be as in  Lemma \ref{psi1-positive-lem}.  Then there exist constants $\xi_0\ge\xi_1$ and $\tau_2>\tau_1$ such that $\psi_1$  is a subsolution of \eqref{l0-eqn0} if \eqref{theta1-2-ineqn11} holds
and $\psi_1$ is a supersolution of \eqref{l0-eqn0} if \eqref{theta1-2-ineqn12} holds
in the region $\{(\eta,\tau):\eta\ge\xi_0e^{-\gamma\tau}, \tau\ge\tau_2\}$. Moreover \eqref{psi1-positive10} holds.
\end{prop}

We now let
\begin{equation*}\label{vkj-defn}
\quad v_{k,j}(\eta)=\eta^{-k-\frac{1}{\gamma}}(\log\,\eta)^{j}\quad\forall \eta>1,0\le j\le k, k\ge 3.
\end{equation*}
and set $v_{k,-1}(\eta)=v_{k,-2}(\eta)=0$ for any $\eta>1$, k$\ge 3$. Then $v_{k,j}$ satisfies 
\begin{equation}\label{vkj-eqn}
(1+k\gamma)v_{k,j}+\gamma\eta (v_{k,j})_{\eta}=j\gamma v_{k,j-1}\quad\forall \eta>1,0\le j\le k, k\ge 3.
\end{equation}

\begin{lem}\label{sub-supersoln-outer-far-field-region-lem2}
Let $\gamma>0$, $\delta_0>0$ and $A>1$. Let $N$ be the smallest integer great than $(1+\gamma^{-1})/2$. Then for any given constants $\{c_{k,0}\}_{3\le k\le 2N}$, there exist  a constant $\tau_4>0$ and constants $\{c_{2k,j}\}_{2\le k\le N, 1\le j\le k}$, $\{c_{2k-1,j}\}_{2\le k\le N, 1\le j\le k}$, such that the function
\begin{align}\label{psi2-defn}
\psi_2(\eta,\tau)=&\phi_0(\eta)+e^{-2\gamma\tau}(\phi_1(\eta)+\theta_1\phi_2(\eta))+e^{-\gamma\tau}\theta_2\phi_3(\eta)
+\sum_{k=2}^Ne^{-2k\gamma\tau}\sum_{j=0}^kc_{2k,j}v_{2k,j}(\eta)\notag\\
&\qquad +\sum_{k=2}^Ne^{-(2k-1)\gamma\tau}\sum_{j=0}^kc_{2k-1,j}v_{2k-1,j}(\eta)
\end{align}
is a subsolution of \eqref{l0-eqn0} if \eqref{theta1-2-ineqn1} holds or a supersolution of \eqref{l0-eqn0} if
\eqref{theta1-2-ineqn2} holds in the region
\begin{equation*}
\{(\eta,\tau):\eta\ge A+\delta_0, \tau\ge\tau_4\}
\end{equation*}
and
\begin{equation}\label{psi2-positive10}
\psi_2(\eta,\tau)>0\quad\forall\eta\ge A+\delta_0,\tau\ge\tau_4.
\end{equation}
\end{lem}
\begin{proof}
Since the proof of this proposition is similar to the proof of  Claim 4.8 of \cite{CDK} and Lemma \ref{sub-supersoln-outer-far-field-region-lem}, we will only sketch its proof here. Let $h$ be given by \eqref{h-defn}. Suppose either \eqref{theta1-2-ineqn1} or \eqref{theta1-2-ineqn2} holds. By \eqref{vkj-eqn} and a direct computation,
\begin{align}\label{L0-psi2-eqn1}
L_0(\psi_2)=&(n-1)\left\{e^{-2\gamma\tau}\left[\left(\frac{\phi_{0,\eta\eta}}{\phi_0}-\frac{\psi_{2,\eta\eta}}{\psi_2}\right)
+\left(\frac{2m-1}{1-m}\right)\left(\frac{\phi_{0,\eta}^2}{\phi_0^2}-\frac{\psi_{2,\eta}^2}{\psi_2^2}
\right)
+\left(\theta_1-\left(\frac{2m-1}{1-m}\right)\right)\frac{\phi_{0,\eta}^2}{\phi_0^2}\right]\right.\notag\\
&\qquad+e^{-\gamma\tau}\left.\left[\left(\frac{n-2-m(n+2)}{1-m}\right)\left(\frac{\phi_{0,\eta}}{\phi_0}-\frac{\psi_{2,\eta}}{\psi_2}\right)+\left(\theta_2-\left(\frac{n-2-m(n+2)}{1-m}\right)\right)\frac{\phi_{0,\eta}}{\phi_0}\right]\right\}\notag\\
&\qquad -\sum_{k=2}^Ne^{-2k\gamma\tau}\sum_{j=1}^kj\gamma c_{2k,j}v_{2k,j-1}-\sum_{k=2}^Ne^{-(2k-1)\gamma\tau}\sum_{j=1}^kj\gamma c_{2k-1,j}v_{2k-1,j-1}.
\end{align}
Let $\tau_4>0$ be a constant to be determined later. By \eqref{phi0-defn}, Lemma \ref{h-decay-infty-lem}, Lemma \ref{phi3-decay-at-infty-lem} and the Taylor theorem, for $\tau_4$ sufficiently large we have
\begin{align}
&\frac{\psi_{2,\eta}}{\psi_2}=\frac{\phi_{0,\eta}}{\phi_0}+o(\eta^{-\frac{1}{\gamma}-1})e^{-\gamma\tau},\label{psi2-psi'-ratio-estimate2}\\
&\frac{\psi_{2,\eta}^2}{\psi_2^2}=\left(\frac{\phi_{0,\eta}}{\phi_0}+o(\eta^{-\frac{1}{\gamma}-1})e^{-\gamma\tau}\right)^2
=\frac{\phi_{0,\eta}^2}{\phi_0^2}+o(\eta^{-\frac{2}{\gamma}-2})e^{-\gamma\tau},\label{psi2-psi'-ratio-estimate1}
\end{align}
and
\begin{align}\label{psi2-psi'-ratio-estimate3}
&\frac{\psi_{2,\eta\eta}}{\psi_2}-\frac{\phi_{0,\eta\eta}}{\phi_0}\notag\\
=&\frac{e^{-2\gamma\tau}h_{\eta\eta}+e^{-\gamma\tau}\theta_2\phi_{3,\eta\eta}+\sum_{k=2}^Ne^{-2k\gamma\tau}\sum_{j=0}^k c_{2k,j}(v_{2k,j})_{\eta\eta}+\sum_{k=2}^Ne^{-(2k-1)\gamma\tau}\sum_{j=0}^k c_{2k-1,j}(v_{2k-1,j})_{\eta\eta}
}{a_0}\notag\\
&\qquad +o\left(\eta^{-\frac{2}{\gamma}-2}\right)e^{-\gamma\tau}
\end{align}
hold for any $\eta\ge A+\delta_0,\tau\ge\tau_2$.
By \eqref{phi0'-phi0-ratio-square-lower-bd}, \eqref{psi2-psi'-ratio-estimate2} and \eqref{psi2-psi'-ratio-estimate1}, for $\tau_4$ sufficiently large, 
\begin{equation}\label{compare-bd4}
\left|\frac{n-2-m(n+2)}{1-m}\right|\left|\frac{\phi_{0,\eta}}{\phi_0}-\frac{\psi_{2,\eta}}{\psi_2}\right|\le\frac{1}{2}\left|\theta_2-\left(\frac{n-2-m(n+2)}{1-m}\right)\right|\frac{\phi_{0,\eta}}{\phi_0}
\end{equation}
and
\begin{equation}\label{compare-bd3}
\left|\frac{2m-1}{1-m}\right|\left|\frac{\phi_{0,\eta}^2}{\phi_0^2}-\frac{\psi_{2,\eta}^2}{\psi_2^2}
\right|
\le\frac{1}{4}\left|\theta_1-\left(\frac{2m-1}{1-m}\right)\right|\frac{\phi_{0,\eta}^2}{\phi_0^2}
\end{equation}
hold for any $\eta\ge A+\delta_0,\tau\ge\tau_4$. Since $N>(1+\gamma^{-1})/2$, by  \eqref{psi2-psi'-ratio-estimate3},
\begin{align}\label{ineqn1}
&(n-1)\left(\frac{\phi_{0,\eta\eta}}{\phi_0}-\frac{\psi_{2,\eta\eta}}{\psi_2}\right)e^{-2\gamma\tau} 
-\sum_{k=2}^Ne^{-2k\gamma\tau}\sum_{j=1}^kj\gamma c_{2k,j}v_{2k,j-1}
-\sum_{k=2}^Ne^{-(2k-1)\gamma\tau}\sum_{j=1}^kj\gamma c_{2k-1,j}v_{2k-1,j-1}\notag\\
=&-\frac{(n-1)}{a_0}\left(e^{-4\gamma\tau}h_{\eta\eta}+\sum_{k=2}^{N-1}e^{-2(k+1)\gamma\tau}\sum_{j=0}^kc_{2k,j}(v_{2k,j})_{\eta\eta}\right)-\sum_{k=2}^Ne^{-2k\gamma\tau}\sum_{j=1}^kj\gamma c_{2k,j}v_{2k,j-1}\notag\\
&\qquad -\frac{(n-1)}{a_0}\left(e^{-3\gamma\tau}\theta_2\phi_{3,\eta\eta}+\sum_{k=2}^{N-1}e^{-(2k+1)\gamma\tau}\sum_{j=0}^kc_{2k-1,j}(v_{2k-1,j})_{\eta\eta}\right)\notag\\
&\qquad -\sum_{k=2}^Ne^{-(2k-1)\gamma\tau}\sum_{j=1}^kj\gamma c_{2k-1,j}v_{2k-1,j-1}
+o\left(\eta^{-\frac{2}{\gamma}-2}\right)e^{-3\gamma\tau}\notag\\
=&-e^{-4\gamma\tau}\left(\frac{(n-1)}{a_0}h_{\eta\eta}+\sum_{j=1}^2j\gamma c_{4,j}v_{4,j-1}\right)
-e^{-3\gamma\tau}\left(\frac{(n-1)\theta_2}{a_0}\phi_{3,\eta\eta}+\sum_{j=1}^2j\gamma c_{3,j}v_{3,j-1}\right)\notag\\
&\qquad -\sum_{k=3}^Ne^{-2k\gamma\tau}\left(\frac{(n-1)}{a_0}\sum_{j=1}^kc_{2k-2,j-1}(v_{2k-2,j-1})_{\eta\eta}
+\sum_{j=1}^kj\gamma c_{2k,j}v_{2k,j-1}\right)\notag\\
&\qquad -\sum_{k=3}^Ne^{-(2k-1)\gamma\tau}\left(\frac{(n-1)\theta_2}{a_0}\sum_{j=1}^kc_{2k-3,j-1}(v_{2k-3,j-1})_{\eta\eta}
+\sum_{j=1}^kj\gamma c_{2k-1,j}v_{2k-1,j-1}\right)\notag\\
&\qquad +o\left(\eta^{-\frac{2}{\gamma}-2}\right)e^{-3\gamma\tau}
\quad\forall\eta\ge A+\delta_0,\tau\ge\tau_4.
\end{align}
By Lemma \ref{h-decay-infty-lem} and Lemma \ref{phi3-decay-at-infty-lem} we can choose constants $c_{3,1}$, $c_{3,2}$, $c_{4,1}$, $c_{4,2}$, such that
\begin{equation}\label{h''-eqn10}
\frac{(n-1)}{a_0}h_{\eta\eta}+\sum_{j=1}^2j\gamma c_{4,j}v_{4,j-1}
=\frac{(n-1)\theta_2}{a_0}\phi_{3,\eta\eta}+\sum_{j=1}^2j\gamma c_{3,j}v_{3,j-1}=o\left(\eta^{-\frac{2}{\gamma}-2}\right)\quad\forall\eta\ge A+\delta_0,\tau\ge\tau_4.
\end{equation} 
Since 
\begin{equation*}
(v_{k,j-1})_{\eta\eta}=\left(k+\frac{1}{\gamma}\right)\left(k+1+\frac{1}{\gamma}\right)v_{k+2,j-1}-(j-1)\left(2k+1+\frac{2}{\gamma}\right)v_{k+2,j-2}+(j-1)(j-2)v_{k+2,j-3},
\end{equation*}
we can iteratively choose  constants  $\{c_{2k,j}\}_{2\le k\le N, 1\le j\le k}$, $\{c_{2k-1,j}\}_{2\le k\le N, 1\le j\le k}$, such that
\begin{equation}\label{sum-estimate1}
\sum_{k=3}^Ne^{-2k\gamma\tau}\left(\frac{(n-1)}{a_0}\sum_{j=1}^kc_{2k-2,j-1}(v_{2k-2,j-1})_{\eta\eta}
+\sum_{j=1}^kj\gamma c_{2k,j}v_{2k,j-1}\right)=0
\end{equation}
and
\begin{equation}\label{sum-estimate2}
\sum_{k=3}^Ne^{-(2k-1)\gamma\tau}\left(\frac{(n-1)\theta_2}{a_0}\sum_{j=1}^kc_{2k-3,j-1}(v_{2k-3,j-1})_{\eta\eta}
+\sum_{j=1}^kj\gamma c_{2k-1,j}v_{2k-1,j-1}\right)=0.
\end{equation}
By  \eqref{L0-psi2-eqn1}, \eqref{ineqn1}, \eqref{h''-eqn10},  \eqref{sum-estimate1} and \eqref{sum-estimate2}, 

\begin{align}\label{L0-psi2-eqn12}
L_0(\psi_2)=&(n-1)\left\{e^{-2\gamma\tau}\left[\left(\frac{2m-1}{1-m}\right)\left(\frac{\phi_{0,\eta}^2}{\phi_0^2}-\frac{\psi_{2,\eta}^2}{\psi_2^2}\right)
+\left(\theta_1-\left(\frac{2m-1}{1-m}\right)\right)\frac{\phi_{0,\eta}^2}{\phi_0^2}\right]\right.\notag\\
&\qquad +e^{-\gamma\tau}\left.\left[\left(\frac{n-2-m(n+2)}{1-m}\right)\left(\frac{\phi_{0,\eta}}{\phi_0}-\frac{\psi_{2,\eta}}{\psi_2}\right)+\left(\theta_2-\left(\frac{n-2-m(n+2)}{1-m}\right)\right)\frac{\phi_{0,\eta}}{\phi_0}\right]\right\}\notag\\
&\qquad+ o\left(\eta^{-\frac{2}{\gamma}-2}\right)e^{-3\gamma\tau}\quad\forall\eta\ge A+\delta_0,\tau\ge\tau_4.
\end{align}
By \eqref{phi0'-phi0-ratio-square-lower-bd} for $\tau_4$ sufficiently large the last error term of \eqref{L0-psi2-eqn12} is bounded above by
\begin{equation}\label{estimate15}
\le\frac{e^{-2\gamma\tau}}{4}\left|\theta_1-\left(\frac{2m-1}{1-m}\right)\right|\frac{\phi_{0,\eta}^2}{\phi_0^2}\quad\forall\eta\ge A+\delta_0,\tau\ge\tau_4.
\end{equation}
By \eqref{compare-bd4}, \eqref{compare-bd3},  \eqref{L0-psi2-eqn12} and \eqref{estimate15},
is a subsolution of \eqref{l0-eqn0} if \eqref{theta1-2-ineqn1} holds or a supersolution of \eqref{l0-eqn0} if
\eqref{theta1-2-ineqn2} holds in the region $\{(\eta,\tau):\eta\ge A+\delta_0, \tau\ge\tau_4\}$.

\end{proof}

Let $A>1$ and
\begin{equation}\label{phi4-defn}
\phi_4(\eta)=\phi_3(\eta)+C_{10}\eta^{-\frac{1}{\gamma}-1}\log\,\eta
\end{equation}
for some constant $C_{10}>0$. By Lemma \ref{phi3-decay-at-infty-lem}
we can choose the constant $C_{10}$ sufficiently large such that 
\begin{equation}\label{phi4-positive}
\phi_4(\eta)\ge\frac{C_{10}}{2}\eta^{-\frac{1}{\gamma}-1}\log\,\eta\quad\forall\eta>A.
\end{equation}
Hence we will assume from now on that $C_{10}$ is chosen such that \eqref{phi4-positive} holds.
By an argument similar to the proof of Lemma \ref{psi1-positive-lem} and Lemma \ref{sub-supersoln-outer-region-near-left-bdary-lem} we have the following result.

\begin{lem}\label{sub-supersoln-outer-region-near-left-bdary-lem2}
Let $\gamma>0$, $\theta_1\in\mathbb{R}$,  $\theta_2\ge 0$ and $A>1$.  Let $N$ be the smallest integer great than $(1+\gamma^{-1})/2$, $\{c_{k,0}\}_{3\le k\le 2N}$ be constants and $\psi_2$ be given by \eqref{psi2-defn} for some constants $\{c_{2k,j}\}_{2\le k\le N, 1\le j\le k}$, $\{c_{2k-1,j}\}_{2\le k\le N, 1\le j\le k}$ given by Lemma \ref{sub-supersoln-outer-far-field-region-lem2}. Then there exist constants $\xi_0>0$,  $\delta_0>0$ and $\tau_2>0$ such that 
$\psi_2$  is a subsolution of \eqref{l0-eqn0} if \eqref{theta1-2-ineqn11} holds
and a supersolution of \eqref{l0-eqn0} if \eqref{theta1-2-ineqn12} holds in the region
\begin{equation*}
\{(\eta,\tau):\xi_0e^{-\gamma\tau}\le\eta -A\le\delta_0, \tau\ge\tau_2\}.
\end{equation*}
Moreover
\begin{equation}\label{psi2-positive6}
\psi_2(\eta,\tau)>0\quad\forall\eta\ge A+\xi_0e^{-\gamma\tau}, \tau\ge\tau_2.
\end{equation}
\end{lem}

By an argument similar to the proof of Lemma \ref{sub-supersoln-outer-far-field-region-lem} and Lemma \ref{sub-supersoln-outer-far-field-region-lem2} respectively we have the following two lemmas.

\begin{lem}\label{sub-supersoln-outer-far-field-region-lem3a}
Let $\gamma>1$, $\delta_0>0$, $\theta_1\in\mathbb{R}$, $\theta_2\ge 0$ and $\phi_4$ be given by \eqref{phi4-defn}.
Let
\begin{equation}\label{psi3-defn}
\psi_3(\eta,\tau)=\phi_0(\eta)+e^{-2\gamma\tau}(\phi_1(\eta)+\theta_1\phi_2(\eta))+e^{-\gamma\tau}\theta_2\phi_4(\eta)\quad\forall\eta>A,\tau\in\mathbb{R}.
\end{equation} 
Then there exists  $\tau_3>0$ such that $\psi_3$  is a subsolution of \eqref{l0-eqn0} if
\eqref{theta1-2-ineqn1} and a supersolution of \eqref{l0-eqn0} if \eqref{theta1-2-ineqn2} holds in the region
\begin{equation*}
\{(\eta,\tau):\eta\ge A+\delta_0, \tau\ge\tau_3\}.
\end{equation*}
\end{lem}

\begin{lem}\label{sub-supersoln-outer-far-field-region-lem4}
Let $\gamma>0$, $\delta_0>0$, $A>1$ and $\phi_4$ be given by \eqref{phi4-defn}. Let $N$ be the smallest integer great than $(1+\gamma^{-1})/2$. Then for any given constants $\{c_{k,0}\}_{3\le k\le 2N}$, there exist constant $\tau_4>0$ and constants $\{c_{2k,j}\}_{2\le k\le N, 1\le j\le k}$, $\{c_{2k-1,j}\}_{2\le k\le N, 1\le j\le k}$, satisfying
\begin{equation}\label{sum-estimate21}
\frac{(n-1)}{a_0}h_{\eta\eta}+\sum_{j=1}^2j\gamma c_{4,j}v_{4,j-1}
=\frac{(n-1)\theta_2}{a_0}\phi_{4,\eta\eta}+\sum_{j=1}^2j\gamma c_{3,j}v_{3,j-1}=o\left(\eta^{-\frac{2}{\gamma}-2}\right)\quad\forall\eta\ge A+\delta_0,\tau\ge\tau_4.
\end{equation} 
and \eqref{sum-estimate1} , \eqref{sum-estimate2},
such that the function
\begin{align}\label{psi4-defn}
\psi_4(\eta,\tau)=&\phi_0(\eta)+e^{-2\gamma\tau}(\phi_1(\eta)+\theta_1\phi_2(\eta))+e^{-\gamma\tau}\theta_2\phi_4(\eta)
+\sum_{k=2}^Ne^{-2k\gamma\tau}\sum_{j=0}^kc_{2k,j}v_{2k,j}(\eta)\notag\\
&\qquad +\sum_{k=2}^Ne^{-(2k-1)\gamma\tau}\sum_{j=0}^kc_{2k-1,j}v_{2k-1,j}(\eta)
\quad\forall\eta>A,\tau\in\mathbb{R}
\end{align}
is a subsolution of \eqref{l0-eqn0} if
\eqref{theta1-2-ineqn1} and a supersolution of \eqref{l0-eqn0} if \eqref{theta1-2-ineqn2} holds in the region
\begin{equation*}
\{(\eta,\tau):\eta\ge A+\delta_0, \tau\ge\tau_4\}
\end{equation*}
and
\begin{equation*}
\psi_4(\eta,\tau)>0\quad\forall\eta\ge A+\delta_0,\tau\ge\tau_4.
\end{equation*}
\end{lem}

Similarly by an argument similar to the proof of Lemma \ref{psi1-positive-lem} and Lemma \ref{sub-supersoln-outer-region-near-left-bdary-lem} we have the following results.

\begin{lem}\label{sub-supersoln-outer-region-near-left-bdary-lem5}
Let $\gamma>0$, $\theta_1\in\mathbb{R}$, $\theta_2\ge 0$ and $\psi_3$ be given by \eqref{psi3-defn}.  Then there exist constants $\xi_0>0$,  $\delta_0>0$ and $\tau_5>0$ such that 
 $\psi_3$  is a subsolution of \eqref{l0-eqn0} if \eqref{theta1-2-ineqn11} holds and is a supersolution of \eqref{l0-eqn0} if
\eqref{theta1-2-ineqn12} holds in the region
\begin{equation*}
\{(\eta,\tau):\xi_0e^{-\gamma\tau}\le\eta -A\le\delta_0, \tau\ge\tau_5\}.
\end{equation*}
Moreover
\begin{equation}\label{psi3-positive}
\psi_3(\eta,\tau)>0\quad\forall\eta\ge A+\xi_0e^{-\gamma\tau}, \tau\ge\tau_5.
\end{equation}
\end{lem}

\begin{lem}\label{sub-supersoln-outer-region-near-left-bdary-lem6}
Let $\gamma>0$, $\theta_1\in\mathbb{R}$,  $\theta_2\ge 0$ and $A>1$.  Let $N$ be the smallest integer great than $(1+\gamma^{-1})/2$, $\{c_{k,0}\}_{3\le k\le 2N}$ be constants and $\psi_4$ be given by \eqref{psi4-defn} for some constants $\{c_{2k,j}\}_{2\le k\le N, 1\le j\le k}$, $\{c_{2k-1,j}\}_{2\le k\le N, 1\le j\le k}$ satisfying \eqref{sum-estimate1}, \eqref{sum-estimate2} and \eqref{sum-estimate21}. Then there exist constants $\xi_0>0$,  $\delta_0>0$ and $\tau_6>0$ such that 
$\psi_4$ is  a subsolution of \eqref{l0-eqn0} if \eqref{theta1-2-ineqn11} holds
and a supersolution of \eqref{l0-eqn0} if \eqref{theta1-2-ineqn12} holds in the region
\begin{equation*}
\{(\eta,\tau):\xi_0e^{-\gamma\tau}\le\eta -A\le\delta_0, \tau\ge\tau_6\}.
\end{equation*}
Moreover
\begin{equation}\label{psi2-positive5}
\psi_4(\eta,\tau)>0\quad\forall\eta\ge A+\xi_0e^{-\gamma\tau}, \tau\ge\tau_6.
\end{equation}
\end{lem}
By Lemma \ref{sub-supersoln-outer-far-field-region-lem2},  Lemma \ref{sub-supersoln-outer-region-near-left-bdary-lem2}, Lemma \ref{sub-supersoln-outer-far-field-region-lem3a}, Lemma \ref{sub-supersoln-outer-far-field-region-lem4}, Lemma \ref{sub-supersoln-outer-region-near-left-bdary-lem5} and Lemma \ref{sub-supersoln-outer-region-near-left-bdary-lem6}, we have the following results.

\begin{prop}\label{sub-supersoln-outer-region-prop2}
Let $\gamma>0$, $\theta_1\in\mathbb{R}$,  $\theta_2\ge 0$ and $A>1$.  Let $N$ be the smallest integer great than $(1+\gamma^{-1})/2$, $\{c_{k,0}\}_{3\le k\le 2N}$ be constants and $\psi_2$ be given by \eqref{psi2-defn} for some constants $\{c_{2k,j}\}_{2\le k\le N, 1\le j\le k}$, $\{c_{2k-1,j}\}_{2\le k\le N, 1\le j\le k}$ satisfying \eqref{sum-estimate1}, \eqref{sum-estimate2} and \eqref{sum-estimate21}. Then there exist constants $\xi_0>0$ and $\tau_2>0$ such that 
$\psi_2$  is a subsolution of \eqref{l0-eqn0} if \eqref{theta1-2-ineqn11}
and a supersolution of \eqref{l0-eqn0} if \eqref{theta1-2-ineqn12} holds in the region
\begin{equation*}
\{(\eta,\tau):\eta\ge A+\xi_0e^{-\gamma\tau}, \tau\ge\tau_2\}.
\end{equation*}
Moreover \eqref{psi2-positive6} holds.
\end{prop}

\begin{prop}\label{sub-supersoln-outer-far-field-region-prop3}
Let $\gamma>1$, $\theta_1\in\mathbb{R}$, $\theta_2\ge 0$, $A>1$ and $\phi_4$, $\psi_3$, be given by \eqref{phi4-defn} and \eqref{psi3-defn} respectively.
Then there exist constants $\xi_0>0$, $\tau_2>0$, such that $\psi_3$  is a subsolution of \eqref{l0-eqn0} if \eqref{theta1-2-ineqn11} and a supersolution of \eqref{l0-eqn0} if \eqref{theta1-2-ineqn12} 
holds in the region
\begin{equation*}
\{(\eta,\tau):\eta\ge A+\xi_0e^{-\gamma\tau}, \tau\ge\tau_2\}.
\end{equation*}
Moreover \eqref{psi3-positive} holds with $\tau_5=\tau_2$.
\end{prop}

\begin{prop}\label{sub-supersoln-outer-far-field-region-prop4}
Let $\gamma>0$, $A>1$ and $\phi_4$, $\psi_4$, be given by \eqref{phi4-defn} and \eqref{psi4-defn} respectively. Let $N$ be the smallest integer great than $(1+\gamma^{-1})/2$. Then for any given constants $\{c_{k,0}\}_{3\le k\le 2N}$, there exist constants $\xi_0>0$, $\tau_2>0$ and constants $\{c_{2k,j}\}_{2\le k\le N, 1\le j\le k}$, $\{c_{2k-1,j}\}_{2\le k\le N, 1\le j\le k}$, satisfying \eqref{sum-estimate1}, \eqref{sum-estimate2} and \eqref{sum-estimate21} such that the function $\psi_4$ is  a subsolution of \eqref{l0-eqn0} if \eqref{theta1-2-ineqn11} and a supersolution of \eqref{l0-eqn0} if \eqref{theta1-2-ineqn12} holds in the region 
\begin{equation*}
\{(\eta,\tau):\eta\ge A+\xi_0e^{-\gamma\tau}, \tau\ge\tau_2\}.
\end{equation*}
Moreover \eqref{psi2-positive5} holds with $\tau_6=\tau_2$.
\end{prop}

\section{Subsolution and supersolution in the domain}
\setcounter{equation}{0}
\setcounter{thm}{0}

In this section we will construct subsolutions and supersolutions of \eqref{L1-eqn} in the inner region using match asymptotic method.
Since the construction is similar to section 6 of \cite{CDK} we will only sketch the argument here.

Let $\lambda>0$ and $\bar{\phi}_0$ be given by \eqref{phi0-bar-defn}. We first recall some results of \cite{Hs1} and \cite{Hs5}.

\begin{thm}\label{self-similar-solution-properties-thm}(cf. Lemma 3.1 of \cite{Hs1}, Theorem 1.1, Theorem 2.1 and proof of  Lemma 2.3 of \cite{Hs5})
Let $n\ge 3$, $0<m<\frac{n-2}{n}$ and $\bar{\phi}_0$ be a solution of \eqref{phi2-0-eqn} given by \eqref{phi0-bar-defn}. Then $\bar{\phi}_0\in C^{\infty}(\mathbb{R})$ and $\bar{\phi}_{0,s}(s)>0$ for any $s>0$. Moreover if $m\ne\frac{n-2}{n+2}$, then  the following holds:
\begin{enumerate}[(i)]

\item 
\begin{equation*}
\bar{\phi}_0(s)=\frac{2(n-1)(n-2-nm)s}{(1-m)\gamma A}-\frac{(n-1)[n-2-m(n+2)]}{(1-m)\gamma A}\log\, s+K_1+o(1)\quad\mbox{ as }s\to\infty
\end{equation*}
for some constant $K_1=K_1\in\mathbb{R}$ depending on $\lambda$, $m$, $n$ and $A$.

\item 
\begin{equation*}
\bar{\phi}_{0,s}(s)=\frac{2(n-1)(n-2-nm)}{(1-m)\gamma A}-\frac{(n-1)[n-2-m(n+2)]}{(1-m)\gamma As}+o(s^{-1})\quad\mbox{ as }s\to\infty.
\end{equation*}
\end{enumerate}

\end{thm}

From now on we will assume $0<m<\frac{n-2}{n+2}$ and
fix $\gamma>0$, $A>1$,
\begin{equation}\label{theta12-=--dfn}
\theta_1^-<\frac{2m-1}{1-m},\quad\theta_2^-=0,\quad\theta_1^+>\max\left(0,\frac{2m-1}{1-m}\right),\quad\theta_2^+>\frac{n-2-m(n+2)}{1-m}.
\end{equation}
We also let $N$ be the smallest integer great than $(1+\gamma^{-1})/2$ and $\{c_{k,0}^{\pm}\}_{3\le k\le 2N}$ be some given constants.
Let the  constants $\{c_{2k,j}^{\pm}\}_{2\le k\le N, 1\le j\le k}$, $\{c_{2k-1,j}\}_{2\le k\le N, 1\le j\le k}$, be given by 
\eqref{sum-estimate1}, \eqref{sum-estimate2} and \eqref{sum-estimate21} with $c_{k,j}=c_{k,j}^{\pm}$ and $\theta_2=\theta_2^{\pm}$.
Let $\psi_3^{\pm}, \psi_4^{\pm}$, be given by \eqref{psi3-defn} and \eqref{psi4-defn} with $\theta_1=\theta_1^{\pm}, \theta_2=\theta_2^{\pm}$, $c_{k,j}=c_{k,j}^{\pm}$  respectively,
\begin{equation*}
\psi^{\pm}(\eta,\tau)=\left\{\begin{aligned}
&\psi_3^{\pm}(\eta,\tau)\quad\mbox{ if }\gamma>1\\
&\psi_4^{\pm}(\eta,\tau)\quad\mbox{ if }0<\gamma\le 1
\end{aligned}\right.
\end{equation*}
and
\begin{equation*}
h^{\pm}=\phi_1+\theta_1^{\pm}\phi_2.
\end{equation*}
Then by Proposition \ref{sub-supersoln-outer-far-field-region-prop3} and  Proposition \ref{sub-supersoln-outer-far-field-region-prop4} there exist constants 
\begin{equation}\label{xi0-ineqn}
\xi_0>\sqrt{\frac{(n-1)|\theta_1^-|}{a_0}}
\end{equation}
 $\tau_2>0$, such that $\psi^+$, $\psi^-$, are supersolution and subsolution of \eqref{l0-eqn0} in the region
$\{(\eta,\tau):\eta\ge A+\xi_0e^{-\gamma\tau}, \tau\ge\tau_2\}$. Moreover
\begin{equation}\label{psi-positive}
\psi^{\pm}>0\quad\forall \eta\ge A+\xi_0e^{-\gamma\tau}, \tau\ge\tau_2.
\end{equation}
For the case $0<\gamma\le 1$ by \eqref{phi4-positive} we can choose $\tau_2$ sufficiently large such that
\begin{equation}\label{theta2+-compare}
\frac{\theta_2^+}{5}\phi_4(\eta)\ge\sum_{k=2}^Ne^{-(2k-1)\gamma\tau}\left|\sum_{j=0}^kc_{2k,j}^{\pm}v_{2k,j}(\eta)\right|
 +\sum_{k=2}^Ne^{-(2k-2)\gamma\tau}\left|\sum_{j=0}^kc_{2k-1,j}^{\pm}v_{2k-1,j}(\eta)\right|\quad\forall\eta>A,\tau\ge\tau_2.
\end{equation}
Then by \eqref{phi2-eqn}, \eqref{phi4-positive},  \eqref{theta12-=--dfn},  \eqref{psi-positive} and \eqref{theta2+-compare},
\begin{equation}\label{psi+-compare}
\psi^+>\psi^->0\quad\forall \eta\ge A+\xi_0e^{-\gamma\tau}, \tau\ge\tau_2.
\end{equation}
Let $\xi_1\ge\xi_0$ be a constant to be determined later.
By \eqref{phi0-bar-defn}, Theorem \ref{self-similar-solution-properties-thm} and the intermediate value theorem, for any  $0\le\varepsilon<1$, $\tau\ge\tau_2$, there exist unique constants $C_{1,\varepsilon}(\tau,\xi_1)$, $C_{2,\varepsilon}(\tau,\xi_1)$, such that
\begin{equation}\label{inner-outer-solution-matching}
e^{\gamma\tau}\psi^+(A+\xi_1e^{-\gamma\tau},\tau)=\frac{\bar{\phi}_0(\xi_1+C_{1,\varepsilon}(\tau,\xi_1)
)}{1+\varepsilon},\quad
e^{\gamma\tau}\psi^-(A+\xi_1e^{-\gamma\tau},\tau)=\frac{\bar{\phi}_0(\xi_1+C_{2,\varepsilon}(\tau,\xi_1))}{1-\varepsilon}.
\end{equation}
When there is no ambiguity we will write $C_{1,\varepsilon}(\tau)$, $C_{2,\varepsilon}(\tau)$, for $C_{1,\varepsilon}(\tau,\xi_1)$, $C_{2,\varepsilon}(\tau,\xi_1)$, respectively. For any $0\le\varepsilon<1$, let
\begin{equation}\label{phi-epsilon-bar-+-defn}
\bar{\phi}_{\varepsilon}^+(\xi,\tau)=\frac{\bar{\phi}_0(\xi+C_{1,\varepsilon}(\tau))}{1+\varepsilon},\quad \tilde{\phi}_{\varepsilon}^-(\xi,\tau)=\frac{\bar{\phi}_0(\xi+C_{2,\varepsilon}(\tau))}{1-\varepsilon}.
\end{equation}
Since by Theorem \ref{self-similar-solution-properties-thm} $\bar{\phi}_0(s)$ is a smooth strictly monotone increasing function, $C_{1,\varepsilon}(\tau)$, $C_{2,\varepsilon}(\tau)$, are smooth function of $\tau\ge\tau_2$.
Let
\begin{align}\label{psi+-defn}
\psi_{\varepsilon}^+(\xi,\tau)=\left\{\begin{aligned}
&\bar{\phi}_{\varepsilon}^+(\xi,\tau)\qquad\qquad\qquad\forall\xi\le\xi_1\\
&e^{\gamma\tau}\psi^+(A+\xi e^{-\gamma\tau},\tau)\quad\,\,\forall\xi>\xi_1
\end{aligned}\right.
\end{align}
and
\begin{align}\label{psi--defn}
\psi_{\varepsilon}^-(\xi,\tau)=\left\{\begin{aligned}
&\bar{\phi}_{\varepsilon}^-(\xi,\tau)\qquad\qquad\qquad\forall\xi\le\xi_1\\
&e^{\gamma\tau}\psi^-(A+\xi e^{-\gamma\tau},\tau)\quad\,\,\forall\xi>\xi_1.
\end{aligned}\right.
\end{align}
Then the following holds.

\begin{lem}\label{psi3-=--sub-supersoln-lem}
$\psi_{\varepsilon}^{\pm}\in C(\mathbb{R}\times (\tau_2,\infty))\cap C^{\infty}((\mathbb{R}\setminus\{\xi_1\})\times (\tau_2,\infty))$ and 
\begin{equation*}
L_1(\psi_{\varepsilon}^+)>0>L_1(\psi_{\varepsilon}^-)\quad\mbox{ in }(\xi_1,\infty)\times (\tau_2,\infty)
\end{equation*} 
where $L_1$ is given by \eqref{l1-defn}.
\end{lem}
We will prove that for sufficiently large $\xi_1$ there exists $\tau_0>\tau_2$ such that $\psi_{\varepsilon}^+(\xi,\tau)$ and $\psi_{\varepsilon}^-(\xi,\tau)$ are supersolution and subsolution of \eqref{L1-eqn} in the region $(-\infty,\xi_1)\times (\tau_0,\infty)$. We first observe that since $\bar{\phi}_0(s)$ is a smooth strictly monotone increasing function of $s$, by \eqref{psi+-compare} and \eqref{inner-outer-solution-matching} we have the following result.

\begin{lem}\label{c1-c2-epsilon-monotonicity-lem}
The following holds:

\begin{enumerate}[(i)]

\item
\begin{equation*}
C_{1,\varepsilon}(\tau)>C_{2,\varepsilon}(\tau)\quad\forall 0\le\varepsilon<1,\tau\ge\tau_2
\end{equation*}

\item
\begin{equation*}
C_{1,\varepsilon_1}(\tau)>C_{1,\varepsilon_2}(\tau)>C_{1,0}(\tau)\quad\forall 0<\varepsilon_2<\varepsilon_1<1,\tau\ge\tau_2
\end{equation*}

\item
\begin{equation*}
C_{2,\varepsilon_1}(\tau)<C_{2,\varepsilon_2}(\tau)<C_{2,0}(\tau)\quad\forall 0<\varepsilon_2<\varepsilon_1<1,\tau\ge\tau_2
\end{equation*}

\item
\begin{equation*}
\lim_{\varepsilon\to 0}C_{1,\varepsilon}(\tau)=C_{1,0}(\tau)\quad\forall \tau\ge\tau_2
\end{equation*}

\item
\begin{equation*}
\lim_{\varepsilon\to 0}C_{2,\varepsilon}(\tau)=C_{2,0}(\tau)\quad\forall \tau\ge\tau_2
\end{equation*}

\end{enumerate}
\end{lem}

\begin{lem}\label{psi-limit-property-lem}
For any $\xi_1\ge\xi_0$, the following holds.

\begin{enumerate}[(i)]

\item 
\begin{equation*}
\lim_{\tau\to\infty}\frac{e^{\gamma\tau}\psi^+(A+\xi_1e^{-\gamma\tau},\tau)}{\tau}=\frac{(n-1)\theta_2^+}{A}
\end{equation*}

\item
\begin{equation*}
\lim_{\tau\to\infty}e^{\gamma\tau}\psi^-(A+\xi_1e^{-\gamma\tau},\tau)=\frac{a_0}{\gamma A}\xi_1+\frac{(n-1)\theta_1^-}{\gamma A\xi_1}
\end{equation*}

\item
\begin{equation*}
\lim_{\tau\to\infty}\left.\frac{\partial}{\partial\xi}\left[e^{\gamma\tau}\psi^{\pm}(A+\xi e^{-\gamma\tau},\tau)\right]\right|_{\xi=\xi_1}
=\frac{a_0}{\gamma A}-\frac{(n-1)\theta_2^{\pm}}{\gamma A\xi_1}-\frac{(n-1)\theta_1^{\pm}}{\gamma A\xi_1^2}
\end{equation*}
\end{enumerate}
where $a_0$ is given by \eqref{a0-defn}.
\end{lem}
\begin{proof}
By \eqref{phi0-defn}, \eqref{phi4-defn}, \eqref{psi3-defn}, \eqref{psi4-defn}, Lemma \ref{h-decay-near-a-lem} and Lemma \ref{phis-near-A-expansion-lem},
\begin{align*}
&\lim_{\tau\to\infty}\frac{e^{\gamma\tau}\psi^+(A+\xi_1e^{-\gamma\tau},\tau)}{\tau}\notag\\
=&\lim_{\tau\to\infty}\frac{e^{\gamma\tau}\phi_0(A+\xi_1e^{-\gamma\tau})}{\tau}
+\lim_{\tau\to\infty}\frac{e^{-\gamma\tau}h^+(A+\xi_1e^{-\gamma\tau})}{\tau}
+\theta_2^+\lim_{\tau\to\infty}\frac{\phi_3(A+\xi_1e^{-\gamma\tau})}{\tau}\\
=&a_0\xi_1\gamma\lim_{z\to 0}\frac{A^{\frac{1}{\gamma}}(A+z)^{-\frac{1}{\gamma}}-1}{z\log\,(z/\xi_1)}
+\lim_{\tau\to\infty}\frac{(n-1)\theta_1^+}{\gamma A\xi_1\tau}
+\frac{(n-1)\theta_2^+}{\gamma A}\lim_{\tau\to\infty}\frac{\gamma\tau+\log\,(1/\xi_1)}{\tau}\notag\\
=&-\frac{a_0\xi_1/A}{\lim_{z\to 0}\log\,(z/\xi_1)}+\frac{(n-1)\theta_2^+}{A}\notag\\
=&\frac{(n-1)\theta_2^+}{A},
\end{align*}
\begin{align*}
\lim_{\tau\to\infty}e^{\gamma\tau}\psi^-(A+\xi_1e^{-\gamma\tau},\tau)
=&\lim_{\tau\to\infty}e^{\gamma\tau}\phi_0(A+\xi_1e^{-\gamma\tau})
+\lim_{\tau\to\infty}e^{-\gamma\tau}h^-(A+\xi_1e^{-\gamma\tau})\\
=&a_0\xi_1\lim_{z\to 0}\frac{1-A^{\frac{1}{\gamma}}(A+z)^{-\frac{1}{\gamma}}}{z}
+\frac{(n-1)\theta_1^-}{\gamma A\xi_1}\notag\\
=&\frac{a_0}{\gamma A}\xi_1+\frac{(n-1)\theta_1^-}{\gamma A\xi_1}
\end{align*}
and
\begin{align*}
&\lim_{\tau\to\infty}\left.\frac{\partial}{\partial\xi}\left[e^{\gamma\tau}\psi^{\pm}(A+\xi e^{-\gamma\tau},\tau)\right]\right|_{\xi=\xi_1}\notag\\
=&\lim_{\tau\to\infty}\psi_{\eta}^{\pm}(A+\xi_1 e^{-\gamma\tau},\tau)\notag\\
=&\lim_{\tau\to\infty}\phi_{0,\eta}(A+\xi_1e^{-\gamma\tau})
+\lim_{\tau\to\infty}e^{-2\gamma\tau}h_{\eta}^{\pm}(A+\xi_1e^{-\gamma\tau})
+\theta_2^{\pm}\lim_{\tau\to\infty}e^{-\gamma\tau}\phi_{3,\eta}^{\pm}(A+\xi_1e^{-\gamma\tau})\notag\\
=&\frac{a_0}{\gamma A}-\frac{(n-1)\theta_2^{\pm}}{\gamma A\xi_1}-\frac{(n-1)\theta_1^{\pm}}{\gamma A\xi_1^2}
\end{align*}
and the lemma follows.
\end{proof}

\begin{lem}\label{c1-2-epsilon-derivative-bd-lem}
For any $\xi_1\ge\xi_0$, $0\le\varepsilon\le 1/2$, there exists a constant $M_{\varepsilon}=M_{\varepsilon}(\xi_1)>0$ such that
\begin{equation}\label{c1-2-derivative-bd}
|C_{1,\varepsilon}'(\tau)|\le M_{\varepsilon}\quad\mbox{ and }\quad |C_{2,\varepsilon}'(\tau)|\le M_{\varepsilon}\quad\forall\tau\ge\tau_2.
\end{equation}
Moreover there exists a constant $M_1=M_1(\xi_1)>0$ such that
\begin{equation}\label{c1-epsilon-uniform-bd}
|C_{2,\varepsilon}(\tau)|\le M_1\quad\forall\tau\ge\tau_2, 0\le\varepsilon\le 1/2.
\end{equation}
\end{lem}
\begin{proof}
Let $\xi_1\ge\xi_0$.
By (i) and (ii) of Lemma \ref{psi-limit-property-lem}, \eqref{xi0-ineqn} and \eqref{inner-outer-solution-matching},
\begin{equation}\label{c1--epsilon-bd1}
\lim_{\tau\to\infty}\frac{\bar{\phi}_0(\xi_1+C_{1,\varepsilon}(\tau))}{(1+\varepsilon)\tau}=\lim_{\tau\to\infty}\frac{\bar{\phi}_{\varepsilon}^+(\xi_1,\tau)}{\tau}=\frac{(n-1)\theta_2^+}{A}\quad\forall 0\le\varepsilon\le 1/2
\end{equation}
and
\begin{equation}\label{c2--epsilon-bd1}
\lim_{\tau\to\infty}\frac{\bar{\phi}_0(\xi_1+C_{2,\varepsilon}(\tau))}{1-\varepsilon}=\frac{a_0}{\gamma A}\xi_1+\frac{(n-1)\theta_1^-}{\gamma A\xi_1}>0\quad\forall 0\le\varepsilon\le 1/2.
\end{equation}
Since $\bar{\phi}_0(\xi)$ is a strictly monotone increasing function of $\xi\in\mathbb{R}$, by \eqref{c1--epsilon-bd1}, \eqref{c2--epsilon-bd1}  and Lemma \ref{c1-c2-epsilon-monotonicity-lem}, there exist constants $x_0, x_1\in\mathbb{R}$ such that 
\begin{equation}\label{c-2-epsilon-lower-bd}
C_{1,\varepsilon}(\tau)\ge C_{1,0}(\tau)\ge x_0\quad\mbox{ and }\quad x_0\le C_{2,1/2}(\tau)\le C_{2,\varepsilon}(\tau)\le C_{2,0}(\tau)\le x_1\quad \forall\tau\ge\tau_2, 0\le\varepsilon\le 1/2
\end{equation}
and \eqref{c1-epsilon-uniform-bd} follows.
Then by Theorem \ref{self-similar-solution-properties-thm} and \eqref{c-2-epsilon-lower-bd} there exist constants $C_1>0$, $C_2>0$, such that
\begin{equation}\label{phi0-bar-derivative-lower-bd}
\bar{\phi}_{0,\eta}(\xi_1+C_{1,\varepsilon}(\tau))\ge C_1\quad\mbox{ and }\quad C_1\le \bar{\phi}_{0,\eta}(\xi_1+C_{2,\varepsilon}(\tau))\le C_2\quad\forall\tau\ge\tau_2, 0\le\varepsilon\le 1/2.
\end{equation}
Let $0\le\varepsilon\le 1/2$.
Differentiating the first term of \eqref{inner-outer-solution-matching} with respect to $\tau$ and letting $\tau\to\infty$, by \eqref{phi0-defn}, Lemma \ref{h-decay-near-a-lem} and Lemma \ref{phis-near-A-expansion-lem} we get,
\begin{align}\label{c2-epsilon-derivaitve-bd1}
\lim_{\tau\to\infty}\frac{\bar{\phi}_{0,\eta}(\xi_1+C_{1,\varepsilon}(\tau))C_{1,\varepsilon}'(\tau)}{1+\varepsilon}
=&\gamma\lim_{\tau\to\infty} [e^{\gamma\tau}\phi_0(A+\xi_1e^{-\gamma\tau})-\xi_1\phi_{0,\eta}(A+\xi_1e^{-\gamma\tau})]\notag\\
&\quad -\gamma\lim_{\tau\to\infty} [e^{-\gamma\tau}h^+(A+\xi_1e^{-\gamma\tau})+\xi_1e^{-2\gamma\tau}h_{\eta}^+(A+\xi_1e^{-\gamma\tau})]\notag\\
&\quad -\gamma\xi_1\theta_2^+\lim_{\tau\to\infty} e^{-\gamma\tau}\phi_{3,\eta}(A+\xi_1e^{-\gamma\tau})\notag\\
=&\frac{(n-1)\theta_2^+}{A}.
\end{align}
By \eqref{phi0-bar-derivative-lower-bd} and \eqref{c2-epsilon-derivaitve-bd1}, there exists a constant $M_{1,\varepsilon}>0$ such that
\begin{equation}\label{c1-epsilon-derivative-bd2}
|C_{1,\varepsilon}'(\tau)|\le M_{1,\varepsilon}\quad\forall\tau\ge\tau_2.
\end{equation}
Similarly there exists a constant $M_{2,\varepsilon}>0$ such that
\begin{equation}\label{c2-epsilon-derivative-bd2}
|C_{2,\varepsilon}'(\tau)|\le M_{2,\varepsilon}\quad\forall\tau\ge\tau_2.
\end{equation}
By \eqref{c1-epsilon-derivative-bd2} and \eqref{c2-epsilon-derivative-bd2}  the lemma follows.
\end{proof}

By Theorem \ref{self-similar-solution-properties-thm}, Lemma \ref{c1-2-epsilon-derivative-bd-lem} and an argument similar to the proof of Proposition 5.1 of \cite{CDK} we have the following result.

\begin{prop}\label{super-sup-solution-prop}
For any $0\le\varepsilon<1/2$ and $\xi_1\ge\xi_0$ there exists a constant $\tau_3=\tau_3(\varepsilon,\xi_1)\ge\tau_2$ such that $\psi_{\varepsilon}^+$ ($\psi_{\varepsilon}^-$ respectively) is a supersolution (subsolution, respectively) of \eqref{L1-eqn} in the region $(-\infty,\xi_1)\times (\tau_3,\infty)$.
\end{prop}

\begin{lem}\label{inner-outer-bdary-derivative-compare-lem}
There exists  a constant $\xi_2\ge\xi_0$ such that for any $\xi_1\ge\xi_2$ there exist constants  $\tau_4=\tau_4(\xi_1)\ge\tau_2$ and $\varepsilon_1=\varepsilon_1(\xi_1)\in (0,1/4)$ such that the following holds.

\begin{enumerate}[(i)]

\item
\begin{equation*}
\lim_{\xi\to\xi_1-}\frac{\partial}{\partial\xi}\psi_{\varepsilon}^+(\xi,\tau)>\lim_{\xi\to\xi_1+}\frac{\partial}{\partial\xi}\psi_{\varepsilon}^+(\xi,\tau)\quad\forall \tau\ge\tau_4, 0\le\varepsilon<\varepsilon_1
\end{equation*}

\item
\begin{equation*}
\lim_{\xi\to\xi_1-}\frac{\partial}{\partial\xi}\psi_{\varepsilon}^-(\xi,\tau)<\lim_{\xi\to\xi_1+}\frac{\partial}{\partial\xi}\psi_{\varepsilon}^-(\xi,\tau)\quad\forall \tau\ge\tau_4, 0\le\varepsilon<\varepsilon_1
\end{equation*}

\item

\begin{equation*}\label{c1-epsilon-upper-bd}
-2\varepsilon\xi_1<C_{2,\varepsilon}(\tau,\xi_1)<\xi_1/4\quad\forall\tau\ge\tau_4, 0\le\varepsilon<1/4.
\end{equation*}

\end{enumerate}

\end{lem}
\begin{proof} 
By \eqref{theta12-=--dfn} and Theorem \ref{self-similar-solution-properties-thm} there exists a constant 
\begin{equation}\label{xi4-defn}
\xi_3\ge\max\left(\xi_0,2e,\frac{4(|\theta_1^-|+1)}{n-2-m(n+2)}\right)
\end{equation}
such that 
\begin{equation}\label{phi0-bar-upper-bd0}
\frac{a_0}{\gamma A}\xi-\frac{2(n-1)[n-2-m(n+2)]}{(1-m)\gamma A}\log\,\xi\le\bar{\phi}_0(\xi)\le\frac{a_0}{\gamma A}\xi-\frac{(n-1)[n-2-m(n+2)]}{2(1-m)\gamma A}\log\,\xi
\end{equation}
for any $\xi\ge\xi_3$,
\begin{equation}\label{phi0-bar-lower-bd2}
\frac{a_0}{\gamma A}-\frac{(n-1)[n-2-m(n+2)]}{2(1-m)\gamma A\xi}>\bar{\phi}_{0,\xi}(\xi)>\frac{a_0}{\gamma A}-\frac{(n-1)\left[\theta_2^++\frac{n-2-m(n+2)}{1-m}\right]}{2\gamma A\xi}\quad\forall\xi\ge\xi_3,
\end{equation}
and
\begin{equation}\label{xi-expression-ineqn}
\frac{2(n-1)|\theta_1^-|+1}{\xi}+\frac{2(n-1)[n-2-m(n+2)]}{(1-m)}\log\,(5\xi/4)<\frac{a_0}{4}\xi\quad\forall \xi\ge\xi_3.
\end{equation}
Let  $\xi_2=2\xi_3$ and $\xi_1\ge\xi_2$. Let 
\begin{equation}\label{xi-expression-tau-compare}
\tau_5>\max \left(\tau_2,\frac{a_0\xi_1}{\gamma\theta_2^+}\right).
\end{equation}
By \eqref{xi4-defn} and \eqref{phi0-bar-upper-bd0},
\begin{align}\label{phi0-bar-upper-bd2a}
\bar{\phi}_0((1-2\varepsilon)\xi_1)
\le&\frac{a_0(1-2\varepsilon)}{\gamma A}\xi_1-\frac{(n-1)[n-2-m(n+2)]}{2(1-m)\gamma A}\log\,[(1-2\varepsilon)\xi_1]\quad\forall 0\le\varepsilon<1/4
\end{align}
and
\begin{equation}\label{phi0-bar-upper-bd2}
\bar{\phi}_0(5\xi_1/4)
\ge\frac{5a_0}{4\gamma A}\xi_1-\frac{2(n-1)[n-2-m(n+2)]}{(1-m)\gamma A}\log\,(5\xi_1/4).
\end{equation}
By Lemma \ref{psi-limit-property-lem}  there exists a constant $\tau_4=\tau_4(\xi_1)\ge\tau_5$ such that

\begin{equation}\label{bdary-estimate1}
e^{\gamma\tau}\psi^+(A+\xi_1e^{-\gamma\tau},\tau)>\frac{(n-1)\theta_2^+}{2A}\tau\quad\forall\tau\ge\tau_4,
\end{equation}
\begin{equation}\label{bdary-estimate2}
\frac{a_0\xi_1}{\gamma A}-\frac{2(n-1)|\theta_1^-|+1}{\gamma A\xi_1}<e^{\gamma\tau}\psi^-(A+\xi_1e^{-\gamma\tau},\tau)<\frac{a_0\xi_1}{\gamma A}+\frac{2(n-1)|\theta_1^-|+1}{\gamma A\xi_1}\quad\forall\tau\ge\tau_4,
\end{equation}
and
\begin{equation}\label{right-derivative-bd}
\left\{\begin{aligned}
&\left.\frac{\partial}{\partial\xi}\left[e^{\gamma\tau}\psi^+(A+\xi e^{-\gamma\tau},\tau)\right]\right|_{\xi=\xi_1}
<\frac{a_0}{\gamma A}-\frac{(n-1)\theta_2^+}{\gamma A\xi_1}\quad\forall\tau\ge\tau_4\\
&\left.\frac{\partial}{\partial\xi}\left[e^{\gamma\tau}\psi^-(A+\xi e^{-\gamma\tau},\tau)\right]\right|_{\xi=\xi_1}
>\frac{a_0}{\gamma A}-\frac{(n-1)|\theta_1^-|+1}{\gamma A\xi_1^2}
\quad\forall\tau\ge\tau_4.
\end{aligned}\right.
\end{equation}
Since $\bar{\phi}_0$ is a strictly monotonce increasing function, by \eqref{inner-outer-solution-matching}, \eqref{xi-expression-tau-compare}, \eqref{phi0-bar-upper-bd2a} and \eqref{bdary-estimate1},

\begin{align}\label{c1-epsilon-positive}
&\bar{\phi}_0(\xi_1+C_{1,\varepsilon}(\tau,\xi_1))>\bar{\phi}_0(\xi_1)\quad\forall\tau\ge\tau_4\notag\\
\Rightarrow\quad&C_{1,\varepsilon}(\tau,\xi_1)\ge 0\quad\forall\tau\ge\tau_4
\end{align}
and by \eqref{inner-outer-solution-matching}, \eqref{xi4-defn}, \eqref{xi-expression-ineqn},  \eqref{phi0-bar-upper-bd2a}, \eqref{phi0-bar-upper-bd2} and \eqref{bdary-estimate2},
\begin{align}\label{ineqn11}
&\bar{\phi}_0((1-2\varepsilon)\xi_1)<\bar{\phi}_0(\xi_1+C_{2,\varepsilon}(\tau,\xi_1))<\bar{\phi}_0(5\xi_1/4)\quad\forall\tau\ge\tau_4\notag\\
\Rightarrow\quad&(1-2\varepsilon)\xi_1<\xi_1+C_{2,\varepsilon}(\tau,\xi_1)<5\xi_1/4\qquad\qquad\qquad\forall\tau\ge\tau_4
\end{align}
and (iii) follows.
Then by  \eqref{psi+-defn}, \eqref {psi--defn}, \eqref{phi0-bar-lower-bd2}, \eqref{c1-epsilon-positive} and \eqref{ineqn11} we have
\begin{align}\label{psi-epsilon-derivative-lower-bd3}
\lim_{\xi\to\xi_1-}\frac{\partial}{\partial\xi}\psi_{\varepsilon}^+(\xi,\tau)
=&\frac{\bar{\phi}_{0,\xi}(\xi_1+C_{1,\varepsilon}(\tau))}{1+\varepsilon}\notag\\
>&\frac{1}{1+\varepsilon}\left(\frac{a_0}{\gamma A}-\frac{(n-1)\left[\theta_2^++\frac{n-2-m(n+2)}{1-m}\right]}{2\gamma A(\xi_1+C_{1,\varepsilon}(\tau))}\right)\notag\\
>&\frac{1}{1+\varepsilon}\left(\frac{a_0}{\gamma A}-\frac{(n-1)\left[\theta_2^++\frac{n-2-m(n+2)}{1-m}\right]}{2\gamma A\xi_1}\right)\quad\forall\tau\ge\tau_4
\end{align}
and
\begin{align}\label{psi-epsilon-derivative-upper-bd4}
\lim_{\xi\to\xi_1-}\frac{\partial}{\partial\xi}\psi_{\varepsilon}^-(\xi,\tau)
=&\frac{\bar{\phi}_{0,\xi}(\xi_1+C_{2,\varepsilon}(\tau))}{1-\varepsilon}\notag\\
<&\frac{1}{1-\varepsilon}\left(\frac{a_0}{\gamma A}-\frac{(n-1)[n-2-m(n+2)]}{2(1-m)\gamma A(\xi_1+C_{2,\varepsilon}(\tau))}\right)\notag\\
<&\frac{1}{1-\varepsilon}\left(\frac{a_0}{\gamma A}-\frac{2(n-1)[n-2-m(n+2)]}{5(1-m)\gamma A\xi_1}\right)\quad\forall\tau\ge\tau_4.
\end{align}
Since  by \eqref{theta12-=--dfn},
\begin{align*}
\lim_{\varepsilon\to 0}\frac{1}{1+\varepsilon}\left(\frac{a_0}{\gamma A}-\frac{(n-1)\left[\theta_2^++\frac{n-2-m(n+2)}{1-m}\right]}{2\gamma A\xi_1}\right)
=&\frac{a_0}{\gamma A}-\frac{(n-1)\left[\theta_2^++\frac{n-2-m(n+2)}{1-m}\right]}{2\gamma A\xi_1}\\
>&\frac{a_0}{\gamma A}-\frac{(n-1)\theta_2^+}{\gamma A\xi_1},
\end{align*}
and by \eqref{xi4-defn},
\begin{align*}
\lim_{\varepsilon\to 0}\frac{1}{1-\varepsilon}\left(\frac{a_0}{\gamma A}-\frac{2(n-1)[n-2-m(n+2)]}{5(1-m)\gamma A\xi_1}\right)
=&\frac{a_0}{\gamma A}-\frac{2(n-1)[n-2-m(n+2)]}{5(1-m)\gamma A\xi_1}\\
<&\frac{a_0}{\gamma A}-\frac{(n-1)|\theta_1^-|+1}{\gamma A\xi_1^2},
\end{align*}
there exists $\varepsilon_1=\varepsilon_1(\xi_1)\in (0,1/4)$ such that for any $0\le\varepsilon<\varepsilon_1$,
\begin{equation}\label{ineqn16}
\frac{1}{1+\varepsilon}\left(\frac{a_0}{\gamma A}-\frac{(n-1)\left[\theta_2^++\frac{n-2-m(n+2)}{1-m}\right]}{2\gamma A\xi_1}\right)
>\frac{a_0}{\gamma A}-\frac{(n-1)\theta_2^+}{\gamma A\xi_1},
\end{equation}
and
\begin{equation}\label{ineqn12}
\frac{1}{1-\varepsilon}\left(\frac{a_0}{\gamma A}-\frac{2(n-1)[n-2-m(n+2)]}{5(1-m)\gamma A\xi_1}\right)
<\frac{a_0}{\gamma A}-\frac{(n-1)|\theta_1^-|+1}{\gamma A\xi_1^2},
\end{equation}
By \eqref{psi+-defn}, \eqref{psi--defn}, \eqref{right-derivative-bd}, \eqref{psi-epsilon-derivative-lower-bd3}, \eqref{psi-epsilon-derivative-upper-bd4}, \eqref{ineqn16} and \eqref{ineqn12}, we get (i) and (ii) and the lemma follows.
\end{proof}

\begin{lem}\label{sub-supersolution-compare-lem2}
Let $\xi_2$ be as in Lemma \ref{inner-outer-bdary-derivative-compare-lem}. Then for any $\xi_1\ge\xi_2$ there exist constants $\tau_5=\tau(\xi_1)\ge\tau_2$ and  $\varepsilon_2=\varepsilon_2(\xi_1)\in (0,1/4)$ such that for any $0\le\varepsilon<\varepsilon_2$,
\begin{equation}\label{psi-epsilon-+compare10}
\psi_{\varepsilon}^+(\xi,\tau)>\psi_{\varepsilon}^-(\xi,\tau)>0\quad\forall \xi\in\mathbb{R},\tau\ge\tau_5.
\end{equation}
\end{lem}
\begin{proof}
By \eqref{psi+-compare} and the definition of $\psi_{\varepsilon}^{\pm}$,
\begin{equation}\label{outer-region-compare}
\psi_{\varepsilon}^+(\xi,\tau)>\psi_{\varepsilon}^-(\xi,\tau)>0\quad\forall \xi\ge\xi_1, \tau\ge\tau_2.
\end{equation}
Since $\bar{\phi}_0$ is a strictly monotone increasing function, by Lemma \ref{c1-c2-epsilon-monotonicity-lem} and the definition of $\bar{\phi}_{\varepsilon}^{\pm}$, for any $0<\varepsilon<1$, $\xi\le\xi_1$, $\tau\ge\tau_2$,
\begin{equation}\label{phi-epsilon-bar-+-ineqn}
\left\{\begin{aligned}
&\bar{\phi}_{\varepsilon}^+(\xi,\tau)>\frac{\bar{\phi}_0(\xi+C_{1,0}(\tau))}{1+\varepsilon}\\
&\bar{\phi}_{\varepsilon}^-(\xi,\tau)<\frac{\bar{\phi}_0(\xi+C_{2,0}(\tau))}{1-\varepsilon}
\end{aligned}\right.
\end{equation}
and by \eqref{inner-outer-solution-matching} and (i) of Lemma \ref{psi-limit-property-lem} there exists a constant $\tau_5=\tau(\xi_1)\ge\tau_2$ such that
\begin{equation}\label{c10-lower-bd5}
C_{1,0}(\tau)=C_{1,0}(\tau,\xi_1)\ge 2\xi_1\quad\forall\tau\ge\tau_5.
\end{equation}
Let $v_0$ be the unique radially symmetric solution of \eqref{elliptic-eqn} and
\begin{equation}\label{v1-v2-defn}
v_1(r)=e^{\frac{4\xi_1}{1-m}}v_0(e^{2\xi_1}r)\quad\mbox{ and }\quad v_2(r)=e^{\frac{\xi_1}{2(1-m)}}v_0(e^{\frac{\xi_1}{4}}r)\quad\forall r\ge 0.
\end{equation}
Then $v_1(0)>v_2(0)$ and by \eqref{phi0-bar-defn} and \eqref{v1-v2-defn},
\begin{equation}\label{v1>v2}
v_1(r)=e^{-\frac{2\xi}{1-m}}\bar{\phi}_0(\xi+2\xi_1)^{\frac{1}{1-m}}>e^{-\frac{2\xi}{1-m}}\bar{\phi}_0(\xi+(\xi_1/4))^{\frac{1}{1-m}}=v_2(r)\quad\forall r=e^{\xi}>0.
\end{equation}
Hence
\begin{equation}\label{v1>v2-2}
v_1(r)>v_2(r)\quad\forall r=e^{\xi}\le e^{\xi_1}\quad
\Rightarrow\quad\min_{r\le e^{\xi_1}}\frac{v_1(r)}{v_2(r)}>1.
\end{equation}
Thus by \eqref{v1>v2} and \eqref{v1>v2-2} there exists $\varepsilon_2=\varepsilon_2(\xi_1)\in (0,1/4)$ such that
\begin{align}\label{compare50}
&\frac{v_1(r)^{1-m}}{v_2(r)^{1-m}}>\frac{1+\varepsilon}{1-\varepsilon}\quad\forall r\le e^{\xi_1}, 0\le\varepsilon<\varepsilon_2\notag\\
\Rightarrow\quad&\frac{\bar{\phi}_0(\xi+2\xi_1)}{1+\varepsilon}>\frac{\bar{\phi}_0(\xi+(\xi_1/4))}{1-\varepsilon}
\quad\forall \xi\le \xi_1, 0\le\varepsilon<\varepsilon_2.
\end{align}
By \eqref{c10-lower-bd5},  \eqref{compare50} and (iii) of Lemma \ref{inner-outer-bdary-derivative-compare-lem},
\begin{equation}\label{compare51}
\frac{\bar{\phi}_0(\xi+C_{1,0}(\tau))}{1+\varepsilon}>\frac{\bar{\phi}_0(\xi+C_{2,0}(\tau))}{1-\varepsilon}
\quad\forall \xi\le \xi_1, \tau\ge\tau_5, 0\le\varepsilon<\varepsilon_2.
\end{equation}
By \eqref{phi-epsilon-bar-+-ineqn} and \eqref{compare51},
\begin{equation}\label{inner-region-compare}
\bar{\phi}_{\varepsilon}^+(\xi,\tau)>\bar{\phi}_{\varepsilon}^-(\xi,\tau)\quad\forall\xi\le\xi_1, \tau\ge\tau_5, 0\le\varepsilon<\varepsilon_2.
\end{equation}
By \eqref{outer-region-compare} and \eqref{inner-region-compare} we get \eqref{psi-epsilon-+compare10}
and the lemma follows.
\end{proof}

\section{Subsolutions, supersolutions and solutions in $\mathbb{R}^n\times (t_0,T)$}
\setcounter{equation}{0}
\setcounter{thm}{0}

In this section we will construct weak subsolutions and supersolutions of \eqref{fde-eqn}. We will then  use these as barriers to construct the unique solution of \eqref{Cauchy-problem} which decays at the rate $(T-t)^{\frac{1+\gamma}{1-m}}$ as $t\nearrow T$.

 We will now let $\xi_2$  be given by  Lemma \ref{inner-outer-bdary-derivative-compare-lem} and let $\xi_1\ge\xi_2$. Let $\tau_4=\tau_4(\xi_1)$, $\tau_5=\tau_5(\xi_1)$,
and $\varepsilon_1=\varepsilon_1(\xi_1),\varepsilon_2=\varepsilon_2(\xi_1)\in (0,1/4)$ be as given by Lemma \ref{inner-outer-bdary-derivative-compare-lem} and Lemma \ref{sub-supersolution-compare-lem2} respectively.
We will fix $0\le\varepsilon<\min (\varepsilon_1,\varepsilon_2)$ and let $\tau_3=\tau_3(\varepsilon,\xi_1)\ge\tau_2$ be given by Proposition \ref{super-sup-solution-prop}. Let $\tau_0=\max (\tau_3, \tau_4, \tau_5)$. Then by Proposition \ref{super-sup-solution-prop} and 
Lemma \ref{inner-outer-bdary-derivative-compare-lem}, $\psi_{\varepsilon}^+$ ($\psi_{\varepsilon}^-$ respectively) is a supersolution (subsolution, respectively) of \eqref{L1-eqn} in the region $(-\infty,\xi_1)\times (\tau_0,\infty)$ and 
(i), (ii), of  Lemma \ref{inner-outer-bdary-derivative-compare-lem} holds for any $\tau\ge\tau_0$. Moreover 
\eqref{psi-epsilon-+compare10} holds in $\mathbb{R}\times (\tau_0,\infty)$.
Let $t_0=T-e^{-\tau_0}$, 
\begin{equation}\label{u-epsilon-+-defn}
u_{\varepsilon}^{\pm}(x,t)=\left\{\begin{aligned}
&\left(\frac{(T-t)^{1+\gamma}}{|x|^2}\psi_{\varepsilon}^{\pm}(\xi,\tau)\right)^{\frac{1}{1-m}}\qquad\qquad\quad\mbox{ if } 0\ne x\in\mathbb{R}^n,\quad\forall t_0\le t<T\\
&\left(\frac{(T-t)^{1+\gamma}e^{2C_{i,\varepsilon}(\tau)-2A(T-t)^{-\gamma}}}{1\pm \varepsilon}\right)^{\frac{1}{1-m}}v_0(0)\quad\mbox{ if }x=0,\quad\forall t_0\le t<T,
\end{aligned}\right.
\end{equation}
\begin{equation*}
\bar{u}_{\varepsilon}^{\pm}(x,t)=\left\{\begin{aligned}
&\left(\frac{(T-t)^{1+\gamma}}{|x|^2}\bar{\phi}_{\varepsilon}^{\pm}(\xi,\tau)\right)^{\frac{1}{1-m}}\qquad\qquad\quad\mbox{ if } 0\ne x\in\mathbb{R}^n,\quad\forall t_0\le t<T\\
&\left(\frac{(T-t)^{1+\gamma}e^{2C_{i,\varepsilon}(\tau)-2A(T-t)^{-\gamma}}}{1\pm \varepsilon}\right)^{\frac{1}{1-m}}v_0(0)\quad\mbox{ if }x=0,\quad\forall t_0\le t<T
\end{aligned}\right.
\end{equation*}
with $i=1$ for $u_{\varepsilon}^+$, $\bar{u}_{\varepsilon}^+$, $i=2$ for $u_{\varepsilon}^-$, $\bar{u}_{\varepsilon}^-$,  and
\begin{equation*}
v^{\pm}(x,t)=\left(\frac{(T-t)^{1+\gamma}}{|x|^2}e^{\gamma\tau}\psi^{\pm}(A+\xi e^{-\gamma\tau},\tau)\right)^{\frac{1}{1-m}}\quad\forall 0\ne x\in\mathbb{R}^n, t_0\le t<T
\end{equation*}
where
\begin{equation*}
\quad\xi=\log |x|-A(T-t)^{-\gamma}, \tau=-\log\, (T-t).
\end{equation*}
Let $r_1(\xi,t)=e^{\xi+A(T-t)^{-\gamma}}$ and $r_1(t)=r_1(\xi_1,t)$. 
Then 
\begin{equation*}
u_{\varepsilon}^{\pm}(x,t)=\left\{\begin{aligned}
&\bar{u}_{\varepsilon}^{\pm}(x,t)\quad\forall |x|\le r_1(t),t_0\le t<T\\
&v^{\pm}(x,t)\quad\forall |x|>r_1(t),t_0\le t<T
\end{aligned}\right.
\end{equation*}
and by Lemma \ref{sub-supersolution-compare-lem2},
\begin{equation}\label{u-epsilon-+compare10}
u_{\varepsilon}^+(x,t)>u_{\varepsilon}^-(x,t)>0\quad\forall x\in\mathbb{R}^n, t_0\le t<T.
\end{equation}
Let
\begin{align*}
&D_1=\{(x,t)\in\mathbb{R}^n\times (t_0,T):|x|<r_1(t)\}\\
&D_2=\{(x,t)\in\mathbb{R}^n\times (t_0,T):|x|>r_1(t)\}\\
&D_1(\xi)=\{(x,t)\in\mathbb{R}^n\times (t_0,T):|x|<r_1(\xi,t)\}\\
&D_2(\xi)=\{(x,t)\in\mathbb{R}^n\times (t_0,T):|x|>r_1(\xi,t)\}
\end{align*}
and
\begin{equation*}
\left\{\begin{aligned}
&\Gamma=\{(x,t)\in\mathbb{R}^n\times (t_0,T):|x|=r_1(t)\}\\
&\Gamma(\xi)=\{(x,t)\in\mathbb{R}^n\times (t_0,T):|x|=r(\xi,t)\}.
\end{aligned}\right.
\end{equation*}
By \eqref{phi0-bar-defn}, \eqref{psi+-defn}, \eqref{psi--defn} and  \eqref{u-epsilon-+-defn},
\begin{equation*}
u_{\varepsilon}^{\pm}(x,t)=\left(\frac{(T-t)^{1+\gamma}e^{2C_{i,\varepsilon}(\tau)-2A(T-t)^{-\gamma}}}{1\pm \varepsilon}\right)^{\frac{1}{1-m}}v_0(|x|^2e^{2C_{i,\varepsilon}(\tau)-2A(T-t)^{-\gamma}})\quad\forall |x|\le e^{\xi_1+A(T-t)^{-\gamma}}, t_0\le t<T
\end{equation*}
with $i=1$ for $u_{\varepsilon}^+$ and $i=2$ for $u_{\varepsilon}^-$.
Hence $u_{\varepsilon}^{\pm}\in C(\mathbb{R}^n\times [t_0,T))$ are smooth functions on $D_1\cup D_2$. By Lemma \ref{psi3-=--sub-supersoln-lem}, Proposition \ref{super-sup-solution-prop}. and the discussion in the introduction section $u_{\varepsilon}^+$ ($u_{\varepsilon}^-$ respectively) is a supersolution (subsolution, respectively) of \eqref{fde-eqn} in the region $D_1\cup D_2$.

\begin{thm}\label{sub-supersolution-thm}
$u_{\varepsilon}^+$ is a weak supersolution of \eqref{fde-eqn}  in $\mathbb{R}^n\times (t_0,T)$ and $u_{\varepsilon}^-$ is a weak subsolution of \eqref{fde-eqn}  in $\mathbb{R}^n\times (t_0,T)$.  
\end{thm}
\begin{proof}
Let $f\in C_0^{\infty}(\mathbb{R}^n\times (t_0,T)$. By the divergence theorem,
\begin{equation}\label{vol-surface-integral-formula1}
\iint_{D_1}\left(-(fu_{\varepsilon}^+)_t+\frac{n-1}{m}\mbox{div}\,(f\nabla (u_{\varepsilon}^+)^m)\right)\,dx\,dt
=\lim_{\xi\to\xi_1-}\iint_{\Gamma(\xi)}\left(\frac{n-1}{m}f\nabla (u_{\varepsilon}^+)^m,-fu_{\varepsilon}^+\right)\cdot\vec{n}_{\xi}\,d\sigma_{\xi} (x,t)
\end{equation}
where $\vec{n}_{\xi}$ is the unit outer normal to the surface $\Gamma(\xi)$ with respect to the domain $D_1(\xi)$ and $d\sigma_{\xi}$ is the surface area element on $\Gamma(\xi)$. Since the left hand side of \eqref{vol-surface-integral-formula1} is equal to 
\begin{align*}
=&\iint_{D_1}\left(-f(u_{\varepsilon}^+)_t-f_tu_{\varepsilon}^++\frac{n-1}{m}\left(f\Delta (u_{\varepsilon}^+)^m+\nabla f\cdot\nabla(u_{\varepsilon}^+)^m\right)\right)\,dx\,dt\notag\\
=&-\iint_{D_1}\left((u_{\varepsilon}^+)_t-\frac{n-1}{m}\Delta (u_{\varepsilon}^+)^m\right)f\,dx\,dt
-\iint_{D_1}\left(u_{\varepsilon}^+f_t+\frac{n-1}{m}(u_{\varepsilon}^+)^m\Delta f\right)\,dx\,dt\notag\\
&\qquad +\frac{n-1}{m}\iint_{\Gamma}(u_{\varepsilon}^+)^m\frac{\partial f}{\partial\vec{n}_{\xi_1}}\,d\sigma (x,t)
\end{align*}
where $d\sigma$ is the surface area element on $\Gamma$, by \eqref{vol-surface-integral-formula1},
\begin{align}\label{vol-surface-integral-formula2}
&-\iint_{D_1}\left((u_{\varepsilon}^+)_t-\frac{n-1}{m}\Delta (u_{\varepsilon}^+)^m\right)f\,dx\,dt
-\iint_{D_1}\left(u_{\varepsilon}^+f_t+\frac{n-1}{m}(u_{\varepsilon}^+)^m\Delta f\right)\,dx\,dt\notag\\
=&\lim_{\xi\to\xi_1-}\iint_{\Gamma(\xi)}\left(\frac{n-1}{m}f\nabla (u_{\varepsilon}^+)^m,-fu_{\varepsilon}^+\right)\cdot\vec{n_{\xi}}\,d\sigma_{\xi} (x,t)-\frac{n-1}{m}\iint_{\Gamma}(u_{\varepsilon}^+)^m\frac{\partial f}{\partial\vec{n}_{\xi_1}}\,d\sigma (x,t).
\end{align}
Similarly
\begin{align}\label{vol-surface-integral-formula3}
&-\iint_{D_2}\left((u_{\varepsilon}^+)_t-\frac{n-1}{m}\Delta (u_{\varepsilon}^+)^m\right)f\,dx\,dt
-\iint_{D_2}\left(u_{\varepsilon}^+f_t+\frac{n-1}{m}(u_{\varepsilon}^+)^m\Delta f\right)\,dx\,dt\notag\\
=&-\lim_{\xi\to\xi_1+}\iint_{\Gamma(\xi)}\left(\frac{n-1}{m}f\nabla (u_{\varepsilon}^+)^m,-fu_{\varepsilon}^+\right)\cdot\vec{n}_{\xi}\,d\sigma_{\xi} (x,t)+\frac{n-1}{m}\iint_{\Gamma}(u_{\varepsilon}^+)^m\frac{\partial f}{\partial\vec{n}_{\xi_1}}\,d\sigma (x,t).
\end{align}
Summing \eqref{vol-surface-integral-formula2} and \eqref{vol-surface-integral-formula3},
\begin{align*}
&-\iint_{D_1\cup D_2}\left((u_{\varepsilon}^+)_t-\frac{n-1}{m}\Delta (u_{\varepsilon}^+)^m\right)f\,dx\,dt
-\iint_{D_1\cup D_2}\left(u_{\varepsilon}^+f_t+\frac{n-1}{m}(u_{\varepsilon}^+)^m\Delta f\right)\,dx\,dt\notag\\
=&\lim_{\xi\to\xi_1-}\iint_{\Gamma(\xi)}\left(\frac{n-1}{m}f\nabla (u_{\varepsilon}^+)^m,-fu_{\varepsilon}^+\right)\cdot\vec{n}_{\xi}\,d\sigma_{\xi} (x,t)\notag\\
&\qquad -\lim_{\xi\to\xi_1+}\iint_{\Gamma(\xi)}\left(\frac{n-1}{m}f\nabla (u_{\varepsilon}^+)^m,-fu_{\varepsilon}^+\right)\cdot\vec{n}_{\xi}\,d\sigma_{\xi} (x,t).
\end{align*}
Hence
\begin{align}\label{vol-surface-integral-formula4}
\iint_{\mathbb{R}^n\times (t_0,T)}\left(u_{\varepsilon}^+f_t+\frac{n-1}{m}(u_{\varepsilon}^+)^m\Delta f\right)\,dx\,dt
=&\iint_{\mathbb{R}^n\times (t_0,T)\setminus\Gamma}\left(\frac{n-1}{m}\Delta (u_{\varepsilon}^+)^m-(u_{\varepsilon}^+)_t\right)f\,dx\,dt
+J_1\le J_1
\end{align}
where
\begin{align}\label{j1-eqn}
J_1=&\lim_{\xi\to\xi_1+}\iint_{\Gamma(\xi)}\left(\frac{n-1}{m}f\nabla (u_{\varepsilon}^+)^m,-fu_{\varepsilon}^+\right)\cdot\vec{n}_{\xi}\,d\sigma_{\xi} (x,t)\notag\\
&\qquad -\lim_{\xi\to\xi_1-}\iint_{\Gamma(\xi)}\left(\frac{n-1}{m}f\nabla (u_{\varepsilon}^+)^m,-fu_{\varepsilon}^+\right)\cdot\vec{n}_{\xi}\,d\sigma_{\xi} (x,t)\notag\\
=&\frac{n-1}{m}\left(\lim_{\xi\to\xi_1+}\iint_{\Gamma(\xi)}\left(\nabla (u_{\varepsilon}^+)^m,0\right)\cdot\vec{n}_{\xi}f\,d\sigma_{\xi} (x,t)-\lim_{\xi\to\xi_1-}\iint_{\Gamma(\xi)}\left(\nabla (u_{\varepsilon}^+)^m,0\right)\cdot\vec{n}_{\xi}f\,d\sigma_{\xi} (x,t)\right)
\end{align}
Now by  \eqref{u-epsilon-+-defn},
\begin{align}\label{normal-derivative-lhs}
&\left(\nabla (u_{\varepsilon}^+(r_1(t)^-,t))^m,0\right)\cdot\vec{n}_{\xi_1}(r_1(t),t)\notag\\
=&\frac{r_1(t)}{E(r_1(t),t)}\frac{\partial}{\partial r}(u_{\varepsilon}^+(r_1(t)^-,t))^m\notag\\
=&\frac{m(T-t)^{\frac{(1+\gamma)m}{1-m}}}{(1-m)r_1(t)^2E(r_1(t),t)}\left(\frac{\psi_{\varepsilon}^+(\xi_1,\tau)}{r_1(t)^2}\right)^{\frac{2m-1}{1-m}}(-2\psi_{\varepsilon}^+(\xi_1,\tau)+\psi_{\varepsilon,\xi}^+(\xi_1^-,\tau))
\end{align}
where 
\begin{equation*}
E(r,t)=\left(r^2+4\gamma^2A^2(T-t)^{-2\gamma-2}e^{4\xi_1+4A(T-t)^{-\gamma}}\right)^{\frac{1}{2}}.
\end{equation*}
Similarly,
\begin{align}\label{normal-derivative-rhs}
&\left(\nabla (u_{\varepsilon}^+(r_1(t)^+,t))^m,0\right)\cdot\vec{n}_{\xi_1}(r_1(t),t)\notag\\
=&\frac{m(T-t)^{\frac{(1+\gamma)m}{1-m}}}{(1-m)r_1(t)^2E(r_1(t),t)}\left(\frac{\psi_{\varepsilon}^+(\xi_1,\tau)}{r_1(t)^2}\right)^{\frac{2m-1}{1-m}}(-2\psi_{\varepsilon}^+(\xi_1,\tau)+\psi_{\varepsilon,\xi}^+(\xi_1^+,\tau)).
\end{align}
By \eqref{j1-eqn}, \eqref{normal-derivative-lhs} and \eqref{normal-derivative-rhs} and Lemma \ref{inner-outer-bdary-derivative-compare-lem},
\begin{equation}\label{j1-negative}
J_1=\frac{n-1}{1-m}\cdot\iint_{\Gamma}\frac{(T-t)^{\frac{(1+\gamma)m}{1-m}}}{r_1(t)^2E(r_1(t),t)}\left(\frac{\psi_{\varepsilon}^+(\xi_1,\tau)}{r_1(t)^2}\right)^{\frac{2m-1}{1-m}}(\psi_{\varepsilon,\xi}^+(\xi_1^+,\tau)-\psi_{\varepsilon,\xi}^+(\xi_1^-,\tau))\,d\sigma (x,t)
\le 0.
\end{equation}
By \eqref{vol-surface-integral-formula4} and \eqref{j1-negative},
\begin{equation*}
\iint_{\mathbb{R}^n\times (t_0,T)}\left(u_{\varepsilon}^+f_t+\frac{n-1}{m}(u_{\varepsilon}^+)^m\Delta f\right)\,dx\,dt
\le 0\quad\forall f\in C_0^{\infty}(\mathbb{R}^n\times (t_0,T).
\end{equation*}
Hence $u_{\varepsilon}^+$ is a weak supersolution of \eqref{fde-eqn}  in $\mathbb{R}^n\times (t_0,T)$. Similarly by Lemma \ref{inner-outer-bdary-derivative-compare-lem} and a similar argument $u_{\varepsilon}^-$ is a weak subsolution of \eqref{fde-eqn} in $\mathbb{R}^n\times (t_0,T)$.
\end{proof}

By an argument similar to the proof of Theorem \ref{sub-supersolution-thm} we have the following result.

\begin{lem}\label{sub-supersoln-bd-domain-lem}
For any $t_0<t_1<T$ and $R>r_1(t_1)$, $u_{\varepsilon}^+$ is a weak supersolution of 
\begin{equation}\label{bd-eqn+}
\left\{\begin{aligned}
&\frac{\partial\zeta}{\partial t}=\frac{n-1}{m}\Delta \zeta^m\quad\mbox{ in }B_R\times (t_0,t_1)\\
&\zeta(x,t)=u_{\varepsilon}^+(x,t)\quad\mbox{ on }\partial B_R\times (t_0,t_1)\\
&\zeta(x,t_0)=u_{\varepsilon}^+(x,t_0)\quad\mbox{ on }B_R
\end{aligned}\right.
\end{equation}
and $u_{\varepsilon}^-$ is a weak subsolution of 
\begin{equation}\label{bd-eqn-}
\left\{\begin{aligned}
&\frac{\partial\zeta}{\partial t}=\frac{n-1}{m}\Delta \zeta^m\quad\mbox{ in }B_R\times (t_0,t_1)\\
&\zeta(x,t)=u_{\varepsilon}^-(x,t)\quad\mbox{ on }\partial B_R\times (t_0,t_1)\\
&\zeta(x,t_0)=u_{\varepsilon}^-(x,t_0)\quad\mbox{ on }B_R.
\end{aligned}\right.
\end{equation} 
\end{lem}

By an argument similar to the proof of Lemma 2.3 of \cite{DaK} we have the following result.
 
\begin{lem}\label{comparison-in-bded-domain-lem}
Let $t_1>t_2>0$, $R>0$ and $Q_R=B_R\times (t_1,t_2)$. Let $g_1,g_2\in C(\partial B_R\times [t_1,t_2)\cup\overline{B_R}\times \{t_1\})$ be such that $g_2\ge g_1\ge 0$ on $\partial B_R\times [t_1,t_2)\cup\overline{B_R}\times \{t_1\}$. Suppose $v_1$, $v_2\in C(\overline{Q_R})$ are weak subsolution and supersolution of \eqref{bd-domain-problem} with $g=g_1,g_2$, respectively and
\begin{equation*}
\min_{\overline{Q_R}}v_i>0\quad\forall i=1,2.
\end{equation*}
Then $v_2(x,t)\ge v_1(x,t)$ for any $(x,t)\in \overline{Q_R}$.
\end{lem}

By Theorem 2.3 of \cite{HP}, Lemma \ref{sub-supersoln-bd-domain-lem}, Lemma \ref{comparison-in-bded-domain-lem} and an argument similar to the of proof of Theorem 1.1 of \cite{Hs2} we have the following result.

\begin{thm}
Let $n\ge 3$ and $0<m<\frac{n-2}{n+2}$. Suppose $u_0$ satisfies
\begin{equation*}
u_{\varepsilon}^-(x,t_0)\le u_0\le u_{\varepsilon}^+(x,t_0)\quad\mbox{ in }\mathbb{R}^n.
\end{equation*}
Then there exists a unique solution of \eqref{Cauchy-problem} which satisfies,
\begin{equation*}
u_{\varepsilon}^-(x,t)\le u(x,t)\le u_{\varepsilon}^+(x,t)\quad\forall x\in\mathbb{R}^n, t_0\le t<T
\end{equation*}
or equivalently,
\begin{equation*}
\psi_{\varepsilon}^-(\xi,\tau)\le\frac{|x|^2u(x,t)^{1-m}}{(T-t)^{1+\gamma}}\le\psi_{\varepsilon}^+(\xi,\tau)\quad\forall 
x\in\mathbb{R}^n, t_0\le t<T
\end{equation*}
where
\begin{equation*}
\xi=\log\,|x|-A(T-t)^{-\gamma},\quad\tau=-\log \,(T-t).
\end{equation*}
\end{thm}

\end{document}